\documentclass[10pt,journal]{IEEEtran}
\usepackage{epsfig,amsmath,amsfonts,cite}
\usepackage{threeparttable}
\usepackage{times}

\usepackage{color}
\usepackage{graphicx}

\usepackage{booktabs}

\usepackage{multirow}

\usepackage{bm}
\usepackage{graphicx}
\usepackage{subfigure}
\usepackage{float}
\usepackage{url}
\usepackage{color}
\usepackage{balance}
\usepackage{multirow}
\usepackage{threeparttable}
\usepackage{stmaryrd}
\usepackage{booktabs}
\usepackage[ruled,vlined]{algorithm2e}
\usepackage{threeparttable}
\usepackage{graphicx}

\DeclareMathAlphabet\mathbfcal{OMS}{cmsy}{b}{n}
\newcommand{\ten}[1]{\mathbfcal{#1}}
\newcommand{\mat}[1]{\mathbf{#1}}

\newcommand{\parm}{{\xi}}
\newcommand{\vecpar}{\boldsymbol{\parm}}
\newcommand{\veceta}{\boldsymbol{\eta}}

\newcommand{\parNum}{d}

\newcommand{\out}{y}

\newcommand{\Phimat}{\boldsymbol{\Phi}}
\newcommand{\multiGPC}{\Psi }

\newcommand{\polyInd}{\alpha}
\newcommand{\basisInd}{\boldsymbol{\polyInd}}
\newcommand{\pcOrder}{p}

\newcommand{\yPC}{\sum\limits_{|\basisInd|=0}^{\pcOrder} {c_{\basisInd}  \multiGPC_{\basisInd}  (\vecpar)} }

\newcommand{\muvec}{\boldsymbol{\mu}}
\newcommand{\Sigmamat}{\mathbf{\Sigma}}

\newcommand{\bvec}{\mathbf{b}}

\newcommand{\yvec}{\mathbf{y}}
\newcommand{\cvec}{\mathbf{c}}

\newcommand{\xvec}{\mathbf{x}}
\newcommand{\Amat}{\mathbf{A}}
\newcommand{\Bmat}{\mathbf{B}}
\newcommand{\Cmat}{\mathbf{C}}

\newcommand{\Mmat}{\mathbf{M}}
\newcommand{\Lmat}{\mathbf{L}}
\newcommand{\Gmat}{\mathbf{G} }
\newcommand{\Atensor}{\mathbfcal{A}}

\newcommand{\Btensor}{\mathbfcal{B}}

\newtheorem{lemma}{Lemma}
\newtheorem{theorem}{Theorem}

\newcommand{\reff}[1]{(\ref{#1})}

\newcommand{\ccf}[1]{\textcolor{black}{#1}}

\def\sssp{\def\baselinestretch{0.88}\large\normalsize}\sssp

\ifCLASSINFOpdf

\else
\fi

\begin{document}

\title{High-Dimensional Uncertainty Quantification of Electronic and Photonic IC with  Non-Gaussian Correlated Process Variations}


\author{Chunfeng Cui and Zheng Zhang,~\IEEEmembership{Member,~IEEE}
\thanks{Some preliminary results of this work have been published in ICCAD 2018~\cite{cui2018uncertainty}.   This work was supported by NSF CAREER Award CCF 1846476, NSF CCF 1763699 and the UCSB start-up grant.}
\thanks{Chufeng Cui and Zheng Zhang are with the Department of Electrical and Computer Engineering, University of California, Santa Barbara, CA 93106, USA (e-mail: chunfengcui@ucsb.edu, zhengzhang@ece.ucsb.edu).}
}


\maketitle

\begin{abstract}

Uncertainty quantification based on generalized polynomial chaos has been used in many applications. It has also achieved great success in variation-aware design automation. However, almost all existing techniques assume that the parameters are mutually independent or Gaussian correlated, which is rarely true in real applications. For instance, in chip manufacturing, many process variations are actually correlated. Recently, some techniques have been developed to handle non-Gaussian correlated random parameters, but they are time-consuming for high-dimensional problems.  
We present a new framework to solve uncertainty quantification problems with many non-Gaussian correlated uncertainties. Firstly, we propose a set of smooth basis functions to well capture the impact of non-Gaussian correlated process variations. We develop a tensor approach to compute these basis functions in a high-dimension setting. Secondly, we investigate the theoretical aspect and practical implementation of a sparse solver to compute the coefficients of all basis functions. We provide some theoretical analysis for the exact recovery condition and error bound of this sparse solver in the context of uncertainty quantification. We present three adaptive sampling approaches to improve the performance of the sparse solver. Finally, we validate our methods by synthetic and practical electronic/photonic ICs with 19 to 57 non-Gaussian correlated variation parameters. Our approach outperforms Monte Carlo by thousands of times in terms of efficiency. It can also accurately predict the output density functions with multiple peaks caused by non-Gaussian correlations, which are hard to capture by existing methods.

\end{abstract}

\section{Introduction}

\IEEEPARstart{U}{ncertainties} are  unavoidable in almost all engineering fields. In semiconductor chip design, a major source of uncertainty is the fabrication process variations. For instance, in deeply scaled electronic integrated circuits (ICs)~\cite{variation2008} and MEMS~\cite{agarwal2009stochastic}, process variations have become a major concern in emerging design technologies such as integrated photonics~\cite{zortman2010silicon}. One of the traditional uncertainty quantification methods is Monte Carlo~\cite{MCintro}, which is easy to implement but has a low convergence rate. In recent years, various stochastic spectral methods (e.g., stochastic Galerkin \cite{ghanem1991stochastic}, stochastic testing \cite{zzhang:tcad2013} and stochastic collocation \cite{xiu2005high}) have been developed and have achieved orders-of-magnitude speedup compared with Monte Carlo in vast applications, including (but not limited to) the modeling and simulation of VLSI interconnects \cite{Tarek_DAC:08, Wang:2004,Shen2010,cmpt2012,chen2014optimal,pham2014decoupled}, nonlinear
ICs~\cite{Strunz:2008, Tao:2007, zzhang:tcad2013, spina2012variability,manfredi2014stochastic,Rufuie2014,ahadi2016sparse}, MEMS~\cite{zzhang_cicc2014,agarwal2009stochastic},  photonic circuits \cite{twweng:optsEx, waqas2018polynomial,melati2015statistical}, and computer architecture \cite{he2019iccad}.

The key idea of stochastic spectral method is to represent the stochastic solution as the linear combination of basis functions. It can obtain highly accurate solutions at a low computational cost when the parameter dimensionality is not high (e.g., less than 20).
Despite their great success, stochastic spectral methods are limited by a long-standing challenge: the generalized polynomial chaos basis functions assume that all random parameters are mutually independent~\cite{gPC2002}. This assumption fails in many realistic cases. For instance, a lot of device-level geometric or electrical parameters are highly correlated because they are influenced by the same fabrication steps; Many circuit-level performance parameters used in system-level analysis depend on each other due to the network coupling and feedback.

Data-processing techniques such as principal or independent component analysis~\cite{wold1987principal,singh2006statistical} can handle Gaussian correlations, but they cause huge errors in general {\it non-Gaussian correlated} cases. Soize and Ghanem  \cite{soize2004physical} proposed to modify the basis functions to a non-smooth chaos formulation, which was applied to the uncertainty analysis of silicon photonics~\cite{twweng:optsEx}. It was found that the method in~\cite{soize2004physical} does not converge well, and designers cannot easily extract mean value and variance from the solution. Recently, we proposed a novel approach to handle non-Gaussian correlated process variations in  \cite{cui2018stochastic_epeps, cui2018stochastic}. We constructed the basis functions via a Gram-Schmidt formula, and then built the surrogate model via an optimization-based stochastic collocation approach.
\ccf{Our} basis functions inherit three important properties of the independent counterparts:  smoothness, orthonormality, and the capability of providing closed-form  mean value and variance of a stochastic solution. In~\cite{cui2018stochastic}, some theoretical results about the numerical error and complexity were provided, and thousands of times of speedup than Monte Carlo were achieved in electronic and photonic ICs with a few non-Gaussian correlated random parameters. \ccf{A later paper \cite{jakeman2019polynomial} presented a similar method to solve the same kind of problems.} However, how to handle high-dimensional non-Gaussian correlated uncertain parameters remains an open question, despite significant progress in high-dimensional uncertainty quantification with independent random parameters~\cite{xli2010, hampton2015compressive,ma2010adaptive,zzhang_cicc2014,el2010variation,zhang2014calculation,zhang2015enabling,zhang2017big,zhang2017tensor}.

{\bf  Contributions.} This paper presents a framework to quantify the uncertainties of electronic and photonic ICs with {\it high-dimensional and non-Gaussian correlated} process variations. Our approach has two excellent features: it efficiently computes high-dimensional basis functions that well capture the impact of non-Gaussian correlated uncertainties; it can also automatically choose informative parameter samples to reduce the numerical simulation cost in a high-dimension setting. The specific contributions of this paper include:
 
\begin{itemize}
\item We derive a class of basis functions for non-Gaussian  correlated random parameters. Our method is based on the {Cholesky factorization} and can overcome the theoretical limitations of~\cite{soize2004physical}. For high-dimensional problems \ccf{that cannot be handled by~\cite{cui2018stochastic_epeps, cui2018stochastic}}, we construct the basis functions via a functional tensor train method when the random parameters are equipped with a Gaussian  mixture density function. We also present a theoretical analysis about the expressive power of our basis functions;
 
\item In order to apply our method to high-dimensional problems, we investigate the theoretical aspect and implementation of sparse solver with $\ell_0$-minimization. Our contributions are twofold. Firstly, we provide the theoretical justification for this $\ell_0$-minimization and an error bound for the resulting surrogate model in the context of uncertainty quantification. Secondly, we improve its performance by adaptive sampling. Instead of using random simulation samples (as done in~\cite{xli2010}), we select the most informative samples via a rank-revealing QR factorization and adaptive optimal sampling criteria including the D-optimal, R-optimal, and E-optimal.
 
\end{itemize}
Compared with our conference paper~\cite{cui2018uncertainty}, this extended journal manuscript presents the following additional results:
\begin{itemize}
    \item  We prove that our basis functions are complete in the polynomial subspace, and that our expression is able to approximate any square-integrable function;
    \item  We show the theoretical conditions to obtain an accurate sparse stochastic approximation and the error  bounds of the resulting sparse stochastic surrogate model;
    \item We proposed two additional approaches, i.e., R-optimal and E-optimal, to select informative samples;
    \item We add more  examples to verify the theoretical properties and performance of our proposed approach.
\end{itemize}

  \begin{figure}[t]
    \centering
        \includegraphics[width=3.6in, height=1.5in]{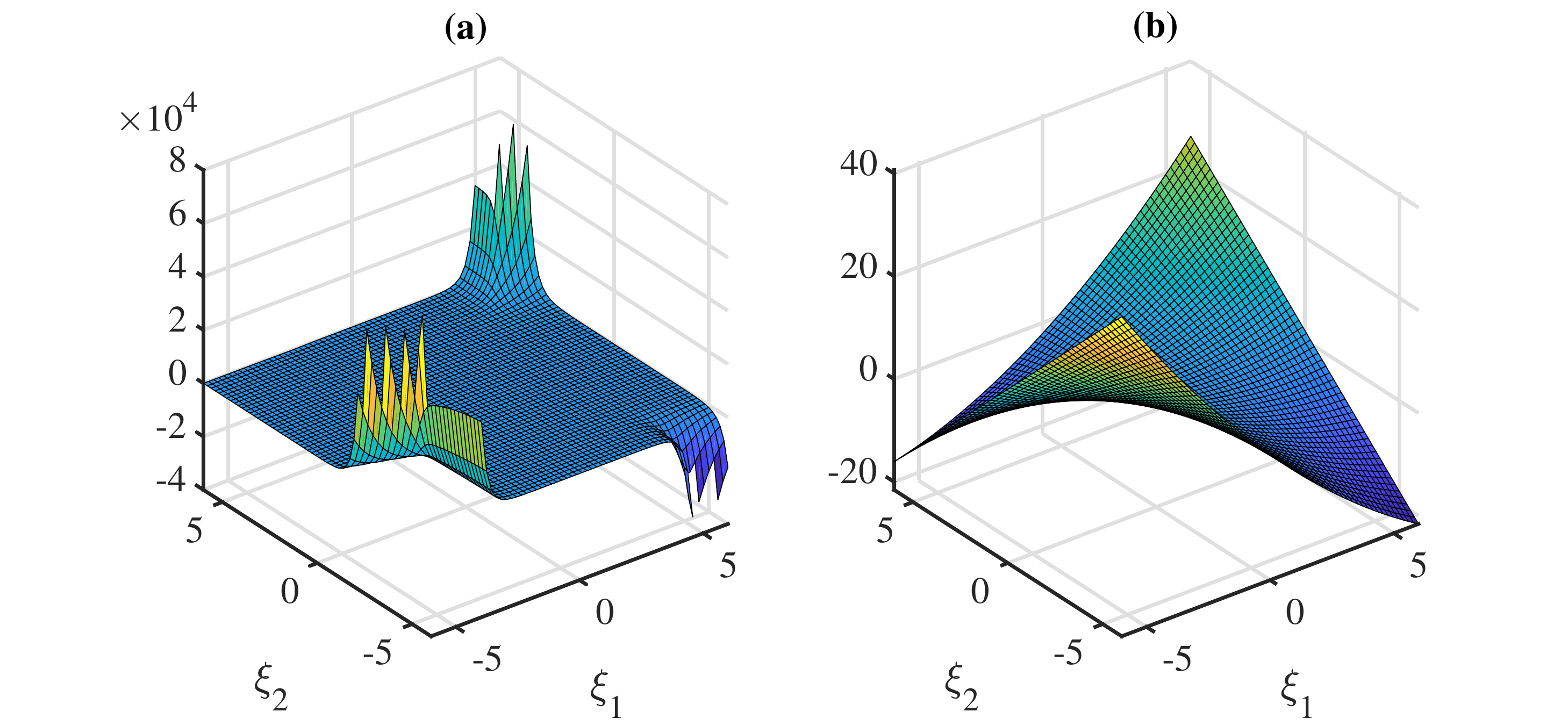}
\caption{ (a): A two-variable basis function by~\cite{soize2004physical}; (b): a basis function obtained by our proposed method.}
    \label{fig:basisfun}
\end{figure}


\section{Preliminary}
\label{sec:preliminaries}

\subsection{Generalized Polynomial Chaos}
\label{subsec:uq}
Let $\vecpar=[\parm_1, \ldots, \parm_{\parNum}] \in \mathbb{R}^{\parNum}$ be $d$ random parameters describing the  process variations. Assume that $\vecpar$ has a joint probability density function $\rho(\vecpar)$, and let $\out (\vecpar) \in \mathbb{R}$ be an uncertain performance metric (e.g., the bandwidth or frequency of a circuit). Suppose $\out (\vecpar)$ is smooth and has a bounded variance. Stochastic spectral method approximates $\out(\vecpar)$ via a truncated generalized polynomial chaos expansion~\cite{gPC2002}:
\begin{equation}
\label{eq:ygpc}
\out (\vecpar) \approx \yPC.
\end{equation}
Here, $c_{\basisInd}$ is the coefficient  and ${\multiGPC}_{\basisInd} \left(\vecpar\right)$ is an orthonormal polynomial satisfying
\begin{equation}
 \mathbb{E}\left[{\multiGPC}_{\basisInd} \left( \vecpar \right) \multiGPC_{\boldsymbol{\beta }}\left( \vecpar \right)\right ]=
\left\{\begin{array}{cc}
1,&{\rm\  if\ } \basisInd=\boldsymbol{\beta };\\
0,&\rm{otherwise.}
\end{array}\right.
\end{equation}
The operator $\mathbb{E}$ denotes expectation;
$\basisInd=[\alpha_1,\ldots, \alpha_{\parNum}] \in \mathbb{N}^{\parNum}$ is a vector  with   $\alpha_i$ being the highest polynomial order in terms of $\xi_i$.
The total polynomial order $|\basisInd|=|\alpha_1|+\ldots +|\alpha_d|$ is bounded by $p$, and thus there are $n=(p+d)!/(p!d!)$  basis functions in total.

The generalized polynomial chaos theory~\cite{gPC2002} assumes that all random parameters are mutually independent. In other words, if $\rho _k(\xi _k)$ denotes the marginal density of $\xi_k$,  the joint density is $\rho(\vecpar) = \prod\limits_{k=1}^d {\rho _k(\xi _k)}$. Under this assumption, a multivariate basis function has a product form:
\begin{equation}
\label{eq:gPC_basis}
\multiGPC_{\basisInd}(\vecpar) =\prod\limits_{k=1}^d {\phi_{k,\alpha_k} ( {\xi_k } )}.
\end{equation}
Here ${\phi_{k,\alpha_k} ( {\xi_k } )}$ is a univariate degree-$\alpha_k$ orthonormal polynomial of parameter $\xi_k$, and it is calculated based on $\rho _k(\xi _k)$ via the three-term recurrence relation~\cite{Walter:1982}.
The unknown coefficients $c_{\basisInd}$ can be computed via various solvers such as stochastic Galerkin \cite{ghanem1991stochastic}, stochastic testing~\cite{zzhang:tcad2013}, and  stochastic collocation~\cite{xiu2005high}. Once $c_{\basisInd}$ are computed, the mean value, variance and density function of $\out (\vecpar)$ can be easily obtained.

However, if the domain is not exactly a tensor product or the parameters are not independent, the above theory cannot be applied directly. This is now an active research topic in both theoretical and application domains.


\subsection{Existing Solutions for Correlated Cases}
  For general non-Gaussian correlated parameters, the reference \cite{soize2004physical} suggested the following basis functions:
\begin{equation}
\label{eq:apc}
\multiGPC_{\basisInd}(\vecpar) =\left(
\prod\limits_{k=1}^d {\rho _k(\xi _k)}/\rho(\vecpar)
\right)^{\frac12}\prod\limits_{k=1}^d {\phi_{k,\alpha_k} ( {\xi_k } )}.
\end{equation}
However, the above basis functions  have two limitations as shown by the numerical results in~\cite{twweng:optsEx}:
\begin{itemize}
\item The basis functions are highly non-smooth and numerically unstable due to the first term  on the right-hand side of \eqref{eq:apc}. This is demonstrated in Fig.~\ref{fig:basisfun}~(a).

\item The basis functions  do not allow an explicit expression for the expectation and variance of $\out (\vecpar)$. This is because the basis function indexed by ${\basisInd}=0$ is not a constant.
\end{itemize}

Recently, we proposed to build a new set of basis functions   via a Gram-Schmidt approach~\cite{cui2018stochastic, cui2018stochastic_epeps}.
We also suggested to compute the coefficients via a stochastic collocation approach $c_{\basisInd}=\mathbb{E}[\multiGPC_{\basisInd}(\vecpar) y(\vecpar)]\approx  \sum_{k} \multiGPC_{\basisInd}(\vecpar_k) y(\vecpar_k)w_k$. A quadrature rule based on an optimization model was developed to compute the quadrature points $\vecpar_k$ and weights $w_k$, and the number of quadrature points was determined  automatically. Our recent technique~\cite{cui2018stochastic, cui2018stochastic_epeps} is highly accurate and efficient for low-dimensional problems, but it suffers from the curse of dimensionality: a huge number of simulation samples will be required if the number of random parameters is large.


\subsection{Background: Tensor Train Decomposition}
\label{subsec:tensortrain}
A  tensor $\ten{A} \in \mathbb{R}^{n_1\times n_2 \times \cdots n_d}$ is a $d$-way data array, which is a high-dimensional generalization of a vector and a matrix. A tensor has  $n_1n_2\ldots n_d$ elements, leading to a prohibitive computation and storage cost. Fortunately, this challenge may be addressed by tensor decomposition~\cite{tensor:suvey}. Among various tensor decomposition approaches, tensor train decomposition~\cite{oseledets2011tensor} is   highly suitable for factorizing high-dimensional tensors. Specifically, given a $d$-way tensor $\Atensor$, the tensor-train decomposition admits a decomposition as
\begin{equation}
a_{i_1 i_2\cdots i_d}=\Amat_1(i_1)\Amat_2(i_{{2}})\ldots \Amat_d(i_d), \forall\, i_k=1,2,\cdots, n_k,
\end{equation}
where $\Amat_k(i_k)$ is an $r_{k-1}\times r_k$ matrix, and
$r_0=r_d=1$.

Given two $d$-way tensors $\Atensor$, $\Btensor$ and their corresponding  tensor train decomposition factors, the tensor train decomposition of
their Hadamard (element-wise) product $\mathcal{C}=\Atensor\circ \Btensor$ has a closed form
\begin{equation}
\label{equ:TTkron}
\Cmat_k(i_k) = \Amat_k(i_k)\otimes \Bmat_k(i_k).
\end{equation}
Here $\otimes$ denotes a matrix Kronecker product.


\section{Basis Functions With Non-Gaussian Correlated Uncertainties}
\label{sec:nongaussian}

Assume that the elements of $\vecpar$ are non-Gaussian correlated. A broad class of non-Gaussian correlated parameters can be described {by} fitting a Gaussian mixture model based on device/circuit testing data. In this general setting, the basis functions in~\eqref{eq:gPC_basis} cannot be employed. Therefore, we will derive a set of multivariate polynomial basis functions, which can be obtained if a multivariate moment computation framework is available. We also show the theoretical completeness and expressive power of our basis functions.


 \subsection{Multivariate Basis Functions}
Several orthogonal polynomials exist for   specific density functions~\cite{ismail2017review}. In general, one may construct   multivariate orthogonal
polynomials via the three-term recurrence
in~\cite{xu1993multivariate} or~\cite{barrio2010three}.
However, their theories   either are hard to implement or can only guarantee weak orthogonality ~\cite{xu1993multivariate, barrio2010three}.
Inspired by  \cite{Golub:1969}, we present a simple yet efficient method for computing  a set of multivariate orthonormal polynomial basis functions.

Let $\vecpar^{\basisInd}=\xi_1^{\alpha_1}\xi_2^{\alpha_2}\ldots\xi_d^{\alpha_d}$ be a \textit{monomial} indexed by $\basisInd$ and the corresponding \textit{moment} be
 \begin{equation}
 \label{equ:moments}
  \mathbb{E}[\vecpar^{\basisInd}]:=\int \vecpar^{\basisInd} \rho(\vecpar) d\vecpar.
 \end{equation}
We resort all monomials bounded by order $p$ in the graded lexicographic order, and denote them as
\begin{equation}\label{equ:monomials}
\bvec(\vecpar)=[b_1(\vecpar),\ldots,b_{n}(\vecpar)]^T\in\mathbb{R}^n.
\end{equation}
Here, $n=\binom{d+p}{d}$.
Further, we denote the  \textit{multivariate moment matrix}  as $\Mmat\in\mathbb{R}^{n\times n}$, where
 \begin{equation}
 \label{equ:mommat}
\mat{M}=\mathbb{E} \left[ \mat{b}(\vecpar) \mat{b}^T(\vecpar)\right], \text{ with } m_{ij}=\mathbb{E}[b_i(\vecpar)b_j(\vecpar)].
 \end{equation}
 {Here, $\mat M$ is also the Gram matrix of $\mat b(\vecpar)$.  Because the monomials $\mat b(\vecpar)$ are linearly independent, it hold that  $\mat M$ is positive definite  according to \cite{horn2012matrix}.}
 For instance, if $d=2$ and $p=1$,   the monomials and the multivariate moment matrix are
 \begin{equation*}
 \mat{b}(\vecpar)=[1, \xi_1,\xi_2]^T \text{ and }\mat{M}=\left[\begin{array}{ccc}
     1 &  \mathbb{E}[\xi_1] &  \mathbb{E}[\xi_2]\\
     \mathbb{E}[\xi_1] &  \mathbb{E}[\xi_1^2] &  \mathbb{E}[\xi_1\xi_2]\\
      \mathbb{E}[\xi_2] &  \mathbb{E}[\xi_1\xi_2] &  \mathbb{E}[\xi_2^2]\\
 \end{array}\right].
 \end{equation*}

We intend to construct $n$ multivariate orthonormal polynomials $\multiGPC_{\basisInd}(\vecpar)$ with their total degrees $|\basisInd|$ bounded by $p$.
{The monomials} $\mat b(\vecpar)$   already contain $n$ different polynomials, even though they are not  orthogonal. The key idea of our method is to  orthogonalize $\mat{b}(\vecpar)$ via a linear projection. This is fulfilled in three steps.
Firstly, we compute all elements in the multivariate moment matrix $\mat{M}$. {This step involves computing the moments up to order $2p$ because $\mat b(\vecpar)$ is up to order $p$.}
Secondly, we decompose $\mat M$ via  the Cholesky factorization $\Mmat=\Lmat\Lmat^T$, where $\mat{L}$ is a   lower-triangular matrix.
{The process is stable when the diagonal elements  $\mathbb{E}[\mat b_i^2(\vecpar)]$ are strictly positive. If $\mat M$ is close to singular, the Cholesky process maybe unstable. In such situations, we add a small permutation term by $\mat M=\mat M + \varepsilon \mat I$.}
Thirdly, we define the orthogonal basis functions as
 \begin{equation}
 \label{equ:basis}
  \boldsymbol{\Psi}(\vecpar) =\Lmat^{-1}\mathbf{b}(\vecpar).
 \end{equation}
Here the $n$-by-1 functional vector $\boldsymbol{\Psi}(\vecpar)$ stores all basis functions $\{\multiGPC_{\basisInd}(\vecpar)\}$ in the graded lexicographic order.  \ccf{This factorization approach can reduce the high-dimensional functional inner-product and moment computations in~\cite{cui2018stochastic_epeps, cui2018stochastic}.}

{\it Properties of the Basis Functions.} Our proposed basis functions have the following excellent properties:
\begin{enumerate}
\item[1)] The basis functions are smooth (c.f. Fig.~\ref{fig:basisfun} (b)). In fact, the basis function $\multiGPC_{\basisInd}(\vecpar)$  is a multivariate polynomial.
It differs from the standard generalized polynomial chaos \cite{gPC2002} in the sense that our basis functions are not the products of univariate polynomials.
\item[2)] All basis functions are orthonormal to each other:
\begin{equation}
 \nonumber
  \mathbb{E}\left[\boldsymbol{\Psi}(\vecpar) \boldsymbol{\Psi}^T(\vecpar)\right]=\Lmat^{-1}\mat{M}\mat{L}^{-T}=\mat{I}.
 \end{equation}
This property is important for the sparse approximation and for extracting the statistical information of $\out(\vecpar)$.
\item[3)]  Since $\multiGPC_{\mat 0}(\vecpar)=1$ is a constant,    the expectation and variance of $\out(\vecpar)$ has a closed-form formula in \eqref{equ:mean}.
\begin{align}
\nonumber \mathbb{E}[\out (\vecpar)] \approx & \sum_{|\basisInd|=0}^p c_{\basisInd} \mathbb{E}\left[\multiGPC_{\basisInd}(\vecpar)\right]=c_{\mat{0}},\\
\text{var}[\out (\vecpar)] = &\mathbb{E} \left[\out^2 (\vecpar) \right]- \mathbb{E}^2 \left[\out (\vecpar) \right] \approx \sum_{|\basisInd|=1}^p c_{\basisInd}^2.
\label{equ:mean}
\end{align}

 \begin{figure*}[t]
    \centering
        \includegraphics[width=5.2in, height=1.5in]{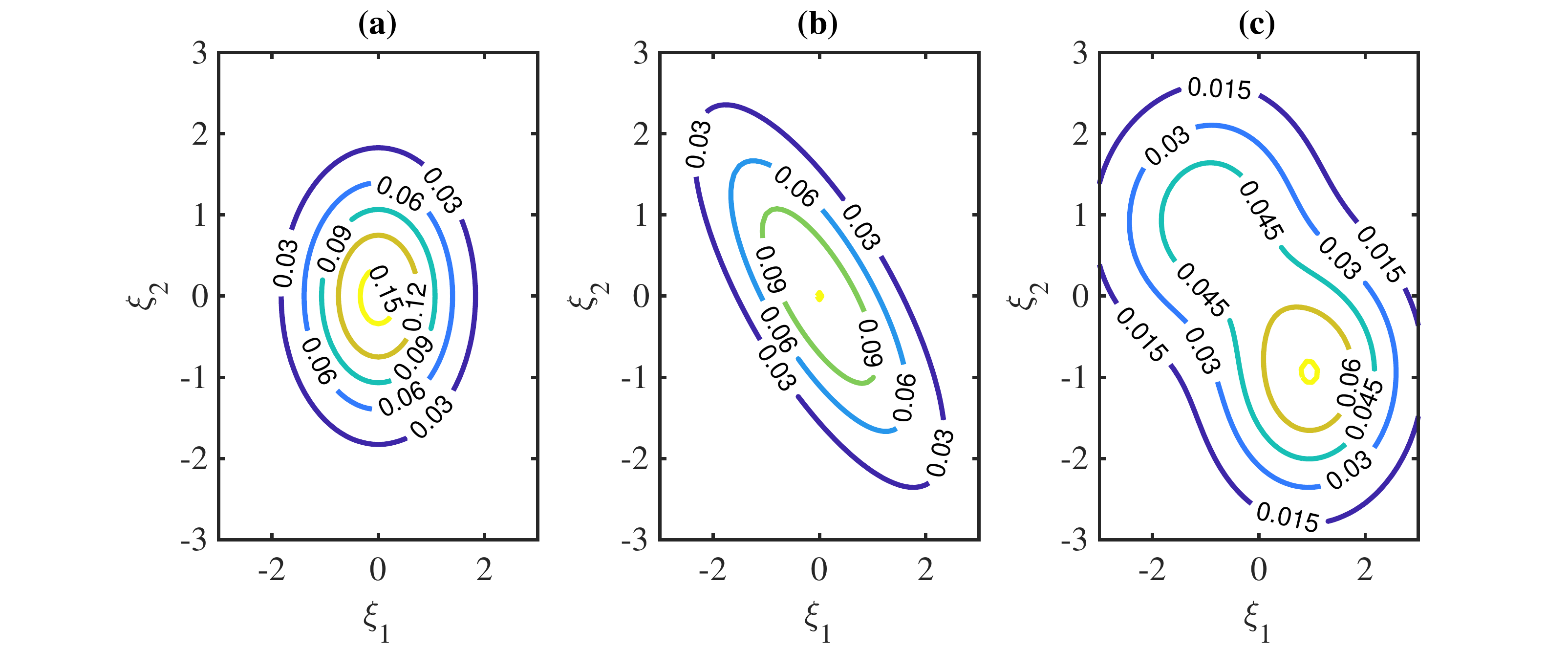}
\caption{Joint density functions for (a): independent Gaussian, (b): correlated Gaussian, (c): correlated non-Gaussian (e.g., a Gaussian  mixture distribution).}
    \label{fig:gmdistribution}
\end{figure*}

\end{enumerate}

\subsection{Theoretical Analysis}

Now we consider the expressive power of our basis functions. This can be described via
\textit{completeness:} a family of basis functions is complete in a space if any function in this space can be uniquely expressed as a linear combination of these basis functions. Denote $\ten{S}_p$ as the space including all polynomials bounded by order $p$. We have the following result.

\begin{lemma}[Completeness]\label{lem:complete}
The  basis functions defined via \eqref{equ:basis} are complete in the space $\ten{S}_p$.
\end{lemma}

\begin{proof}
See Appendix \ref{app:complete}.
\end{proof}

Denote $L^2(\vecpar,\rho(\vecpar))=\{y(\vecpar): \mathbb{E}_{\vecpar}[y^2(\vecpar)]<\infty\}$ as the space of square-integrable functions.
In the following lemma, we  show that our basis function can approximate  any function  in $L^2(\vecpar,\rho(\vecpar){)}$ when the polynomial order $p$ is large enough.

\begin{lemma}[Expressive power]\label{lem:convergence}
Assume that  $\vecpar$ is defined on a compact bounded domain   or there exists a constant $a>0$ such that $\mathbb E[\exp{(a|\xi_i|)}]<\infty$ for $i=1,\ldots,d$, then
\begin{itemize}
\item[(i)] the multivariate polynomials are dense in $ L^2(\vecpar,\rho(\vecpar))$;
\item[(ii)] for any $y(\vecpar)\in L^2(\vecpar,\rho(\vecpar))$, there exists an order-$p$ approximation  $y_p(\vecpar)=\sum_{|\basisInd|=0}^pc_{\basisInd}\multiGPC_{\basisInd}(\vecpar)$, such that $\mathbb{E}[(y(\vecpar)-y_p(\vecpar))^2]\rightarrow 0$ as $p\rightarrow \infty$.
\end{itemize}
\end{lemma}

\begin{proof}
 The detailed proof is given in Appendix \ref{app:convergence}.
\end{proof}

 \section{Higher-Order Moment Computation}
\label{subsec:moments}

We further calculate the $2p$-order moments in order to build the basis functions in a high-dimensional setting.

\subsection{Gaussian Mixture Model}
An excellent choice for the data-driven modeling of $\rho(\vecpar)$ is the Gaussian  mixture model:
 \begin{equation}
 \rho(\vecpar) =\sum_{i=1}^r w_i \mathcal{N}(\vecpar | \muvec_i, \Sigmamat_i), \; {\rm with}\; w_i>0 , \; \sum_{i=1}^r w_i=1.
 \end{equation}
$\mathcal{N}(\vecpar | \muvec_i, \Sigmamat_i)$ denotes the multivariate Gaussian density function with mean $\muvec_i \in\mathbb{R}^d$ and a positive definite covariance matrix $\Sigmamat_i \in \mathbb{R}^{d \times d}$. Fig.~\ref{fig:gmdistribution} has shown the difference of a Gaussian mixture model (e.g., $r>1$) with independent Gaussian (e.g., $r=1$ and $\Sigmamat$ is diagonal) and correlated Gaussian distributions (e.g., $r=1$ and $\Sigmamat$ is not diagonal). The corresponding moment is
\begin{equation}
\label{equ:highmoment}
m_{\basisInd}=\sum\limits_{i=1}^r w_i q_{\basisInd,i}, \; {\rm with}\; q_{\basisInd,i}= \int  \vecpar^{\basisInd} \mathcal{N}(\vecpar | \muvec_i, \Sigmamat_i) d\vecpar. \nonumber
\end{equation}
Existing methods for calculating the higher-order moments for normal distributions rely on the characteristic function \cite{tracy1993higher,phillips2010r}. The main bottleneck of these methods is enumerating an integers matrix.
In this paper, we propose a functional tensor train approach to compute the higher-order moments.

For simplicity, we ignore the subscript index $i$ in $\muvec_i$, $\Sigmamat_i$ and $q_{\basisInd,i}$. Denote $\Amat$ as the lower triangular matrix from the Cholesky decomposition of  $\Sigmamat=\Amat\Amat^T$.  Then $\veceta$ from  $\vecpar=\Amat\veceta+\muvec$  satisfies   $\boldsymbol{\eta}\sim\mathcal{N}(\boldsymbol{\eta}|\mat{0},\mat{I})$. Consequently, $q_{\basisInd}$ can be calculated via
\begin{align}
\label{equ:xi2eta}
\nonumber     q_{\basisInd} =& \int  \vecpar^{\basisInd} \mathcal{N}(\vecpar | \muvec, \Sigmamat)   d\vecpar\\
=&\int  (\Amat\veceta+\muvec)^{\basisInd} \frac{\exp(-\veceta^T\veceta)}{\sqrt{(2\pi)^d}} d\veceta.
\end{align}
The difficulty in computing \reff{equ:xi2eta} lies in  $(\Amat\veceta +\muvec)^{\basisInd}$, which is not the product of univariate functions of each $\eta_i$.

 \subsection{Functional Tensor Train Formula}
Fortunately,  $q_{\basisInd}$ can be computed exactly with an efficient functional tensor-train method.
Specifically, we are seeking for  $\Gmat_0 \in \mathbb{R}^{1\times r_0}$ and a set of univariate functional matrices $ \Gmat_i (\eta_i) \in \mathbb{R}^{r_{i-1} \times r_i}$ for $i=1,\cdots d$   with $r_{d}=1$, such that
\begin{equation}
\label{equ:de-correFun}
(\Amat\veceta +\muvec)^{\basisInd} =\Gmat_0 \Gmat_1(\eta_1)\Gmat_2(\eta_2)\ldots \Gmat_d(\eta_d).
\end{equation}
Afterwards, we can obtain $q_{\basisInd}$ via
\begin{equation}\label{equ:moment}
q_{\basisInd} = \Gmat_0 \mathbb{E}[\Gmat_1(\eta_1)]\mathbb{E}[\Gmat_2(\eta_2)]\ldots \mathbb{E}[\Gmat_d(\eta_d)].
\end{equation}
The detailed derivations of $\mat{G}_i(\eta_i)$ are as follow.

 \subsubsection{Derivation of (14)}

The $j$-th element in    $\vecpar=\Amat\veceta+\muvec$ satisfies
\begin{equation}
\label{equ:xij}
\xi_j=a_{j1}\eta_1+ a_{j2}\eta_2+\ldots+a_{jd}\eta_d+\mu_j, \forall\, j=1,\ldots,d.
\end{equation}
Here $a_{jk}$ denotes the $(j,k)$-th element of $\mat{A}$.

\begin{theorem}[Theorem 2, \cite{oseledets2013constructive}]
Any function satisfies
$$f(x_0,\ldots,x_d)   =\omega_0(x_0)+\ldots+\omega_d(x_d),$$
can be written as a functional tensor train as in Eq.~\eqref{eq:ftt_uni}, which equals to the product of some univariate matrices and vectors.
\label{thm:ftt}
\end{theorem}

\begin{align}
\label{eq:ftt_uni}
f(x_0,x_1,\ldots,x_d)&
=
\left(\begin{array} {cc}
\omega_0(x_0) & 1
\end{array}\right)
\left(\begin{array} {cc}
 1 & 0\\
 \omega_1(x_1)&1
\end{array}\right) \ldots
\left(\begin{array} {cc}
 1&0\\
 \omega_{d-1}(x_{d-1})&1
\end{array}\right)
\left(\begin{array} {c}
 1\\
 \omega_d(x_d)
\end{array}\right).
\end{align}

Applying Theorem~\ref{thm:ftt} to \reff{equ:xij}, we can derive a functional tensor train decomposition for  $\xi_j$:
\small
\begin{equation}
\label{equ:xiTTD}
\xi_j=
\left(\mu_j \ 1\right)
\left(\begin{array} {cc}
 1 & 0\\
 a_{j1} \eta_1&1
\end{array}\right)
\ldots
\left(\begin{array} {cc}
 1&0\\
 a_{j(d-1)} \eta_{d-1}&1
\end{array}\right)
\binom{1}{a_{jd} \eta_d}.
\end{equation}\normalsize
Then the expectation is
\begin{equation}
\label{equ:ExiTTD}
\mathbb{E}[\xi_j]=
\left(\mu_j \ 1\right)
\left(\begin{array} {cc}
 1 & 0\\
 0 &1
\end{array}\right)
\ldots
\left(\begin{array} {cc}
 1&0\\
 0&1
\end{array}\right)
\binom{1}{0} =\mu_j.
\end{equation}
The obtained functional tensor trains \reff{equ:xiTTD} can be reused to compute higher-order moments.

\subsubsection{Recurrence Formula}
For each $\basisInd$ with $1<|\basisInd|\le2p$, there exist $\basisInd_1$ and $\basisInd_2$ with $|\basisInd_1|,|\basisInd_2|\le p$, such that
\begin{equation*}
\vecpar^{\basisInd} = \vecpar^{\basisInd_1}\cdot \vecpar^{\basisInd_2}, {\rm \ where \ }
\basisInd = \basisInd_1+\basisInd_2.
\end{equation*}
According to \reff{equ:TTkron}, the tensor-train representation of
$\vecpar^{\basisInd}$ can be obtained as the {Kronecker}
product of the tensor trains of $\vecpar^{\basisInd_1}$ and $\vecpar^{\basisInd_2}$.
Suppose that $\vecpar^{\basisInd_1}=\mathbf{E}_0(\muvec)\mathbf{E}_1(\eta_1)\cdots \mathbf{E}_{d}(\eta_d)$ and $\vecpar^{\basisInd_2}=\mathbf{F}_0(\muvec)\mathbf{F}_1(\eta_1)\cdots \mathbf{F}_{d}(\eta_d)$, then
\small
\begin{align}
\vecpar^{\basisInd}& =
 \Gmat_0(\muvec)\Gmat_1(\eta_1)\cdots \Gmat_{d}(\eta_d),\; {\rm with \ }\Gmat_i(\eta_i) = \mathbf{E}_i(\eta_i) \otimes \mathbf{F}_i(\eta_i).
 \label{equ:tt_higherorder}
\end{align} \normalsize
Here $\mat G_i(\eta_i)\in\mathbb{R}^{2^{|\basisInd|}\times 2^{|\basisInd|}}$  for all $i=1,\ldots,d-1$.
Because $\eta_i$'s are mutually
independent, finally we have
\begin{equation}
\label{equ:Exi}
\mathbb{E}[\vecpar^{\basisInd}] = \Gmat_0(\muvec)\mathbb{E}[\Gmat_1(\eta_1)]\cdots \mathbb{E}[\Gmat_{d}(\eta_d)].
\end{equation}
In other words, the moments  can be easily computed via small-size matrix-vector  products.

\begin{algorithm}[t]
\label{alg:moments}
\caption{A Functional Tensor-Train Method for Computing Basis Functions of Gaussian Mixture Distributions}
      \SetKwInput{Input}{Input}
      \SetKwInput{Output}{Output}
\Input{The mean value, covariance, and weight  for Gaussian mixtures $\{\muvec_i,\Sigmamat_i, w_i\}_{i=1}^r$, and  order $p$.}
\For{$i=1,\ldots,r$}
{
Compute the Cholesky factor $\Amat$ via $\Sigmamat_i=\Amat\Amat^T$;\\
Calculate the functional tensor trains for the first-order and higher-order monomials via \reff{equ:xiTTD} and \reff{equ:tt_higherorder}, respectively;\\ Obtain the moments via \reff{equ:ExiTTD}
and \reff{equ:Exi}.
}
Assemble the multivariate moment matrix $\Mmat$ in \reff{equ:mommat};\\
Compute the basis functions via \reff{equ:basis}. \\
\Output{The multivariate basis functions $\multiGPC_{\basisInd}(\vecpar)$.}
\end{algorithm}

The basis construction framework is summarized in Alg.~\ref{alg:moments}

\textsl{Remark.}
If  the parameters have a  block-wise correlation structure, we can  divide
$\vecpar=[\xi_1,\ldots,\xi_d]^T$into several disjoint groups  $\vecpar_{g_1},\ldots,\vecpar_{g_r}$ in the following way:  the random parameters inside each group are correlated, while the  parameters among different groups are mutually independent. Under this assumption,
the basis functions can be constructed by
 \begin{equation*}
  \multiGPC_{\basisInd}(\vecpar)=\multiGPC_{\basisInd_1}(\vecpar_{g_1})\ldots\multiGPC_{\basisInd_r}(\vecpar_{g_r}).
\end{equation*}
If there are multiple parameters inside a group, we can construct the basis function $\multiGPC_{\basisInd_i}(\vecpar_{g_i})$ by the proposed formula \reff{equ:basis}. Otherwise, the univariate orthogonal basis functions can be calculated via the three-term recurrence relation~\cite{Walter:1982}.
 
 \section{A Sparse Solver: Why and How Does It Work?}
 \label{sec:sparsesolver}

After constructing the basis functions $\{\multiGPC_{\basisInd}(\vecpar)\}_{|\basisInd|=0}^p$,  we need to compute the weights (or coefficients) $\{c_{\basisInd}\}$. For the independent case, many high-dimensional solvers have been developed, such as compressed sensing~\cite{xli2010, hampton2015compressive}, analysis of variance~\cite{ma2010adaptive,zzhang_cicc2014}, model order reduction~\cite{el2010variation}, hierarchical methods \cite{zhang2014calculation,zzhang_cicc2014,zhang2015enabling}, and tensor computation~\cite{zhang2017big, zhang2015enabling,zhang2017tensor}.
For the non-Gaussian correlated case discussed in this paper, we employ a sparse solver to obtain the coefficients.

For convenience, we resort all basis functions and their weights $\{\multiGPC_{\basisInd}(\vecpar), c_{\basisInd}\}_{|\basisInd|=0}^p$ into  $\{\Psi_j(\vecpar), c_j\}_{j=1}^n$.
Given $m$ pairs of parameter samples and simulation values $\{ \vecpar_k, \out(\vecpar_k)\}_{k=1}^m$,
our task is to find the coefficient $\mat c$ such that
\begin{equation}
\Phimat \cvec = \yvec, \text{\ with\ } \Phi_{kj}=\multiGPC_{j}(\vecpar_k), \ y_k=\out(\vecpar_k),
 \label{equ:linearsystem}
 \end{equation}
where $\Phimat\in\mathbb{R}^{m\times n}$ stores the values of $n$ basis functions at $m$ samples and $\mat y\in\mathbb{R}^{m}$ stores the $m$ simulation values. In practice, computing each sample $\out(\vecpar_k)$ requires calling a  time-consuming device- or circuit-level simulator. Therefore, it is desired to use as few simulation samples as possible.

We consider the compressed sensing technique~\cite{xli2010, hampton2015compressive} with $m\ll n$. We seek for the sparsest solution by solving the $\ell _0$-minimization problem
\begin{equation}
\min_{\cvec\in\mathbb{R}^n} \|\cvec\|_0\quad \text{s.t.}\quad \Phimat\cvec=\yvec.
\label{equ:l0optim}
\end{equation}
Here $\|\cvec\|_0$ denotes the number of nonzero elements.
The compressed sensing technique is subject to some assumptions. Firstly, the solution $\mat{c}$ should be sparse in nature, which is generally true in high-dimensional uncertainty quantification. Secondly, the matrix $\frac{1}{\sqrt{m}}\Phimat$ should satisfy the restricted isometry property (RIP) \cite{candes2006stable}:  there exists $0<\kappa_s<1$ such that
\begin{equation}\label{equ:RIP}
(1-\kappa_s)\|\cvec\|_2^2 \le \frac1m\|\Phimat \cvec\|_2^2\le(1+\kappa_s)\|\cvec\|_2^2
\end{equation}
holds for any $\|\cvec\|_0\le s$.
Here, $\|\cdot\|_2$ is the Euclidean norm. Intuitively, this requires that all columns of $\boldsymbol{\Phi}$ are nearly orthogonal to each other.

Compressed sensing techniques have been extensively studied in signal processing. Now we investigate its theoretical condition and accuracy guarantees in our specific setting: high-dimensional uncertainty quantification with non-Gaussian correlated process variations.

\subsection{Conditions to Achieve RIP}

In general, it is NP-hard to check whether $\Phimat$ satisfies the RIP condition~ \cite{bandeira2013certifying}.
When the number of samples $m$ is large enough,
our matrix $\boldsymbol{\Phi}$  satisfies  $\frac1m\mat{\Phi}^T\mat{\Phi}\approx \mat{I}$, i.e.,
\begin{equation}
\label{eq:rip_test}
\frac{1}{m}\sum \limits_{k=1}^m \left(\Psi_i(\vecpar_k) \Psi_j(\vecpar_k)\right) \approx \mathbb{E}\left[ \Psi_i(\vecpar) \Psi_j(\vecpar) \right]=\delta_{ij}.
\end{equation}
Here $\delta_{ij}$ is the delta function.
Hence the RIP condition \reff{equ:RIP} will be satisfied with  a high probability. The following theorem provides a rigorous guarantee.

\begin{theorem}[Conditions for RIP]\label{thm:MBound}
Denote the random variable $X_k^{ij}=\multiGPC_i(\vecpar_k)\multiGPC_j(\vecpar_k)$. Assume that $X_k^{ij}$ is sub-Gaussian \cite{buldygin1980sub} with variance proxy $\sigma$ for any $i,j,k$, i.e.,
\begin{equation}
\mathbb{E}[\exp(\lambda(X_k^{ij}-\delta_{ij}))]\le \exp(\frac{\sigma^2\lambda^2}{2}),\quad \forall\ \lambda\in\mathbb{R},
\end{equation}
 and the random samples $\{\vecpar_k\}_{k=1}^m$ are generated independently. Then the RIP condition \eqref{equ:RIP} holds with a probability at least $1-\eta$ provided that
\begin{equation}\label{equ:Mbd}
m\ge 2\log\left(2/\eta\right)\frac{s^2\sigma^2}{\kappa_s^2}.
\end{equation}
\end{theorem}

\begin{proof}
 See Appendix \ref{app:Mbound}.
\end{proof}

\subsection{Error Bounds in Uncertainty Quantification}
Under the RIP condition, we are able to approximate the solution with good accuracy. Now we provide the error bounds for $\mat c$ and for the stochastic solution $\mat y(\vecpar)$. In our implementation, we solve the following constrained optimization
\begin{equation}\label{equ:fix_sparse}
    \mat{c}^*=\arg\min\limits_{\mat c}\|\mat{\Phi}\mat{c}-\mat{y}\|_2\quad \text{s.t. } \|\mat{c}\|_0\le s.
\end{equation}
{The above problem can be solved by any $\ell_0$-minimization \ccf{solver}, such as COSAMP \cite{needell2009cosamp}, difference-of-convex \cite{zheng2014successive}, and penalty decomposition \cite{lu2015penalty}.}
The error bound is presented in the following theorem.

\vspace{1mm}
\begin{theorem}[Coefficient error]
\label{thm:accuracy}
Suppose  $\mat{y}=\mat{\Phi c}+\mat{e}$, where $\mat{e}$ is some random noise and $\mat c$ is the exact solution.
Let $\mat{c}_s$ be a sparse vector that remains the $s$ largest-magnitude components of $\mat{c}$ and keeps all other components to be zero, and  $\epsilon=\|\mat{\Phi}\mat{c}^*-\mat{y}\|_2$ is the residue in \reff{equ:fix_sparse}.
If $\mat{\Phi}$ satisfies the $(2s,\kappa_{2s})$-RIP condition, then any solution $\mat{c}^*$ of \reff{equ:fix_sparse} satisfies
\begin{equation*}
    \|\mat{c}-\mat{c}^*\|_2\le\alpha_0\|\mat{c}_s-\mat{c}\|_1+\alpha_1\|\mat{e}\|_2+\alpha_1\epsilon.
\end{equation*}
Here, $\alpha_0=1+\frac{1.7071 \sqrt{1+\kappa_{2s}}}{m(1-\kappa_{2s})\sqrt{s}}$ and $\alpha_1=\frac{1}{m(1-\kappa_{2s})}$ are constants.
\end{theorem}

\begin{proof}
 See Appendix \ref{append:L0error}.
\end{proof}

Theorem \ref{thm:accuracy} shows that numerical error of computing $\mat{c}$ consists of three parts. The first part exists because the  exact solution may not be exactly $s$-sparse. The second part is caused by the numerical errors in device/circuit simulation. The third part is caused by the numerical error in an optimization solver.

Now we consider the error of approximating $y(\vecpar)$. For any square-integrable $y(\vecpar)$,   denote its $\ell_2$ norm as
\begin{equation}
   {\|y(\vecpar)\|_{2}=\sqrt{\mathbb{E}_{\vecpar}[y^2(\vecpar)]}.}
\end{equation}
Further, let $y_p(\vecpar)$ be the projection of $y(\vecpar)$ to the space $\ten{S}_p$. In other words, $y_p(\vecpar)=\sum_{|\basisInd|=0}^pc_{\basisInd}\multiGPC_{\basisInd}(\vecpar)$ is the $p$th-order approximation that we are seeking for. The approximation error is shown as follows.

\vspace{1mm}
\begin{theorem}[Approximation error]
\label{thm:approxerr}
For any square integrable function $y(\vecpar)$, the approximation error is bounded by
\begin{equation}\label{equ:approxerr}
    \|y(\vecpar)-y^*(\vecpar)\|_2 \le\|\mat{c}-\mat{c}^*\|_2 +\|y(\vecpar)-y_p(\vecpar)\|_2,
\end{equation}
{where $y^*(\vecpar)=\sum_{|\basisInd|=0}^p c_{\basisInd}^*\multiGPC_{\basisInd}(\vecpar)$ is our constructed surrogate function}, and $\mat{c}^*=[\cdots, c_{\basisInd}^*, \cdots]$ is the solution of \reff{equ:fix_sparse}.
\end{theorem}
\vspace{1mm}

\begin{proof}
The detailed proof is shown in Appendix \ref{app:approxerr}.
\end{proof}

 For the bound shown in \reff{equ:approxerr}, the first term is caused by the numerical error in computing $\mat{c}$; The second term arises from the distance from  $y(\vecpar)$ to the $p$th-order polynomial space $\ten{S}_p$, which will be sufficiently small if $p$ is large enough.







\begin{algorithm}[t]
\label{alg:AdaptiveSparseSolver}
\caption{An Adaptive Sparse  Solver} 
      \SetKwInput{Input}{Input}
      \SetKwInput{Output}{Output}
\Input{Input a set of candidate samples $\Omega^0$ and  basis functions $\{\multiGPC_j(\vecpar)\}_{j=1}^n$.}

Choose an initial sample set $\Omega\subset \Omega ^0$ via the rank-revealing QR factorization, with $m=|\Omega|\ll n$.      \\
Call the simulator to calculate $\out(\vecpar)$ for all $\vecpar \in \Omega$.

\For{Iteration $t=1,2,\ldots$}
{Solve the $\ell_0$ minimization  problem \reff{equ:fix_sparse} to obtain $\cvec$.\\
Chose the next optimal sample point $\vecpar_{m+1}$ via the D-optimal, R-optimal, or E-optimal criteria. \\
Call the simulator to calculate $\out(\vecpar_{m+1})$.\\  Let $m\leftarrow m+1$. \\
\If{the stopping criterion is satisfied}{Stop}
}
\Output{The  coefficient  $\cvec$ and the surrogate model   $y(\vecpar)=\sum_{|\basisInd|=0}^pc_{\basisInd}\multiGPC_{\basisInd}(\vecpar)$.}
\end{algorithm}

\section{Adaptive Sample Selection}\label{sec:adaptive}

The system equation \eqref{equ:linearsystem} is normally set up by using some random samples~\cite{xli2010}. In practice, some samples are informative, yet others are not. Therefore, we present some adaptive sampling methods to improve the performance of the sparse solver. Our method uses a rank-revealing QR decomposition to pick  initial samples, then it selects subsequent samples using a D-optimal, R-optimal, or E-optimal criterion.
The whole framework is summarized in Alg.~\ref{alg:AdaptiveSparseSolver}.

\subsection{Initial Sample Selection}
We first select a small number of initial samples from a pool of $m_0$ random candidate samples $\Omega^0$. By evaluating the basis functions at all candidate samples, we form a matrix $\Phimat^0\in\mathbb{R}^{m_0\times n}$ whose $j$-th rows stores the values of $n$ basis functions at the $j$-th candidate sample. Here, we use the determinant to measure the importance. Specifically, suppose that the rank of $\mat{\Phi}^0$ is greater than $m$, we want to separate the rows of $\mat{\Phi}$ into two parts via a maximization problem  \begin{equation}\label{equ:D-optimal}
\max_{\mat{P}}\ \det(\mat{\Phi}_a^T\mat{\Phi}_a),\text{ where }\mat{(\Phi^0)}^T\mat P=[\mat{\Phi}_a^T\quad
\mat{\Phi}_b^T].
\end{equation}
We achieve this via a rank-revealing QR factorization~\cite{gu1996efficient}:
\begin{equation}
(\Phimat^0)^T\mathbf{P} = \mathbf{QR},\text{ where }\mat{R}=\left[\begin{array}{cc}
 \mathbf{R}_{11}&\mathbf{R}_{12}\\
 \mathbf{0}&\mathbf{R}_{22}
 \end{array}\right].
\end{equation}
Here£¬ $\mathbf{P}$ is a permutation matrix,  $\mathbf{Q}$ is an orthogonal matrix, and  $\mathbf{R}_{11}\in\mathbb{R}^{m\times m}$ is an upper triangular matrix.
We  compute the permutation matrix
 $\mathbf{P}$  such that  the absolute values of diagonal elements in $\mat{R}$ are in a descending order.
In other words, the first $m$ columns of $\mathbf{P}$ indicate the most informative $m$ rows of $\Phimat^0$. We keep these rows and the associated parameter samples $\Omega$.
The above initial samples have generated a  matrix $\Phimat\in\mathbb{R}^{m\times n}$,  then we further add  new informative samples based on a D-optimal, R-optimal, or E-optimal criterion.

\subsection{D-optimal Adaptive Sampling}

Our first method is motivated by~\cite{prasad2016multidimensional, diaz2017sparse} but differs in the numerical implementation: the method in~\cite{prasad2016multidimensional} defined the candidate sample set as   quadrature nodes, which are unavailable for non-Gaussian correlated case;  the method in~\cite{diaz2017sparse} used rank-revealing QR at each iteration, whereas our method selects new samples via an optimization method.

 \begin{figure*}[t]
    \centering        \includegraphics[width=4.6in]{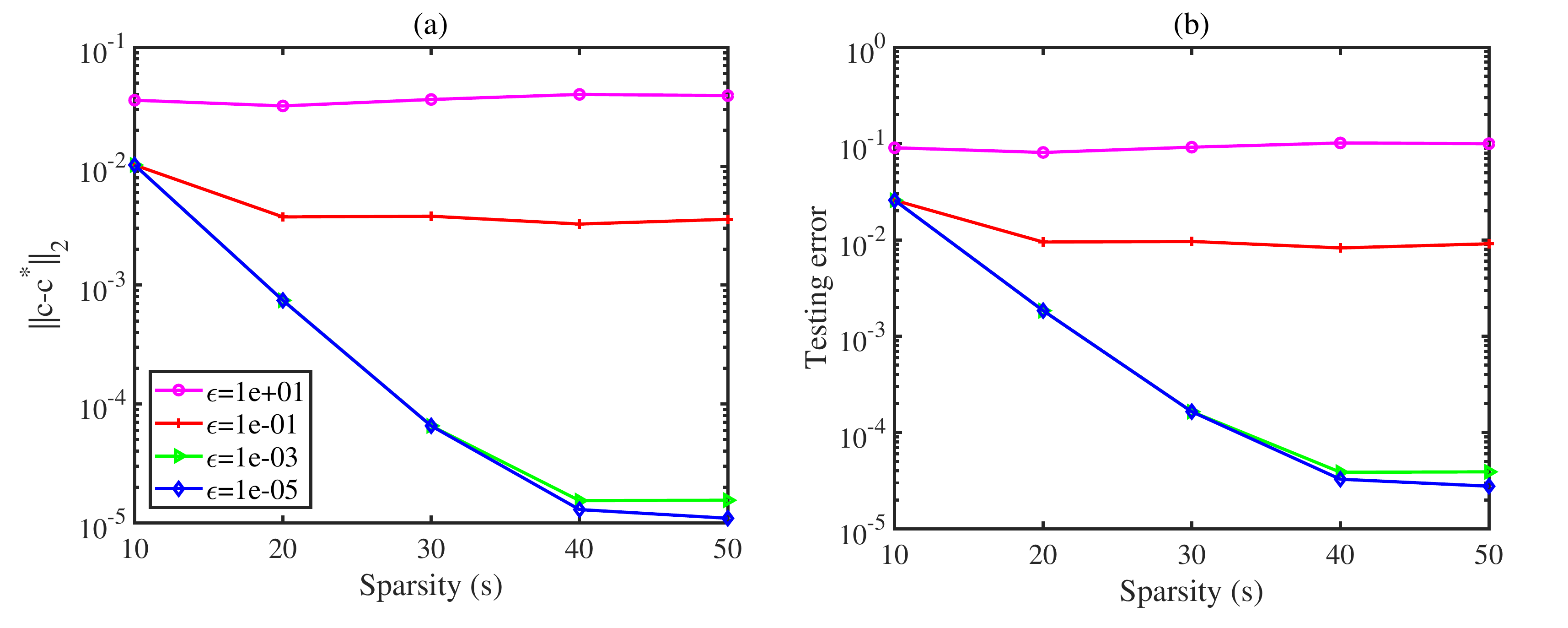}
\caption{Numerical error for the synthetic experiment. (a) Errors in the computed coefficients of our basis functions; (b) Testing errors of approximating $y(\vecpar)$.}
    \label{fig:synethic_err_c}
\end{figure*}

Assume that our sparse solver has computed $\mat{c}$ based on the available samples, then we fix the indices of the nonzero elements in $\mat{c}$ and update the samples and solutions sequentially.
We denote the locations of nonzero coefficients as $\mathcal{C}=\{i_1,\ldots,i_s\}$ with $m>s$, and denote
$\Phimat_s\in\mathbb{R}^{m\times s}$ as the sub-matrix of  $\Phimat$ generated by extracting the columns associated with $\mathcal{C}$.
 The next most informative sample $\vecpar_k$ associated with the row vector $\xvec(\vecpar_k)=[\multiGPC_{i_1} (\vecpar_k) \cdots,\multiGPC_{i_s}(\vecpar_k)]\in\mathbb{R}^{1\times s}$ can be decided via solving the following problem:
\begin{equation}
\label{equ:detmax}
\max_{\vecpar_k\in \Omega^0\setminus \Omega}\quad \det(\Phimat_s^T\Phimat_s + \xvec(\vecpar_k)^T \xvec(\vecpar_k)),
\end{equation}
where $ \Omega^0\setminus \Omega$ includes the sample points in $\Omega^0$ but not in $\Omega$.
It is unnecessary to compute the above determinant for every sample. The matrix determinant lemma \cite{harville1997matrix} shows $\det(\Phimat_s^T\Phimat_s+ \xvec(\vecpar_k)^T\xvec(\vecpar_k)) = \det(\Phimat_s^T\Phimat_s)(1+  \xvec(\vecpar_k)$ $(\Phimat_s^T\Phimat_s)^{-1}\xvec(\vecpar_k)^T)$. Therefore, \reff{equ:detmax} can be solved via \begin{equation}
  \label{equ:easy_detmax}
  \max_{\vecpar_k\in \Omega^0\setminus \Omega}\quad \xvec(\vecpar_k)(\Phimat_s^T\Phimat_s)^{-1}\xvec(\vecpar_k)^T.
  \end{equation}
{In our experiments, we obtain the optimal solution of \reff{equ:easy_detmax} by comparing the  objective values for all sample points in $ \Omega^0\setminus \Omega$.}

   After getting the new sample,
we update the matrix $\Phimat_s:=\left[\begin{array}{c}\Phimat_s\\ \xvec(\vecpar_k)\end{array}\right]$,  update $(\Phimat_s^T\Phimat_s)^{-1}$ via the Sherman-Morrison formula \cite{sherman1950adjustment}, and recompute the $s$ nonzero elements of $\cvec$ by
\begin{equation}
\cvec_1=(\Phimat_s^T\Phimat_s)^{-1}\Phimat_s^T \yvec.
\label{equ:updatec}
\end{equation}
Inspired by \cite{malioutov2010sequential}, we stop the iteration if $\cvec_1$ is close to its previous step or if the maximal iteration number is reached.

\subsection{R-optimal Adaptive Sampling}
The RIP condition \reff{equ:RIP} is equivalent to
\begin{equation}\label{equ:RIP2eig}
    (1-\kappa_s)\le\lambda_{\min}(\frac1m \mat{\Phi}^T_s\mat{\Phi}_s)\le\lambda_{\max}(\frac1m \mat{\Phi}^T_s\mat{\Phi}_s)\le(1+\kappa_s).
\end{equation}
Here $\mat{\Phi}_s$ contains arbitrary $s$ columns of $\mat{\Phi}$. The constraint in \reff{equ:RIP2eig} is equivalent to $\|\frac1m \mat{\Phi}^T_s\mat{\Phi}_s-\mat{I}\|_2\le\kappa_s$.  Therefore, we can select the next sample by minimizing $\|\frac1m \mat{\Phi}^T_s\mat{\Phi}_s-\mat{I}\|_2$. We refer this method as R-optimal  because it optimizes the RIP condition.

Suppose that the locations of nonzero coefficients are fixed as $\mathcal{C}=\{i_1,\ldots,i_s\}$.
 The next sample $\vecpar_k$ associated with the row vector $\xvec(\vecpar_k)=[\multiGPC_{i_1} (\vecpar_k) \cdots,\multiGPC_{i_s}(\vecpar_k)]$ is found via the following optimization problem
\begin{equation}\label{equ:R-opt}
    \min_{\vecpar_k\in\Omega^0\setminus \Omega} \left\|\frac{1}{m+1}\left( \mat{\Phi}^T_s\mat{\Phi}_s+\xvec(\vecpar_k)^T \xvec(\vecpar_k)\right)-\mat{I}\right\|_2.
\end{equation}
{In our implementation, we solve \reff{equ:R-opt} by comparing the objective function values of all available sample points.}

\begin{figure*}[t]
    \centering
        \includegraphics[width=5.0in, height=3.2in]{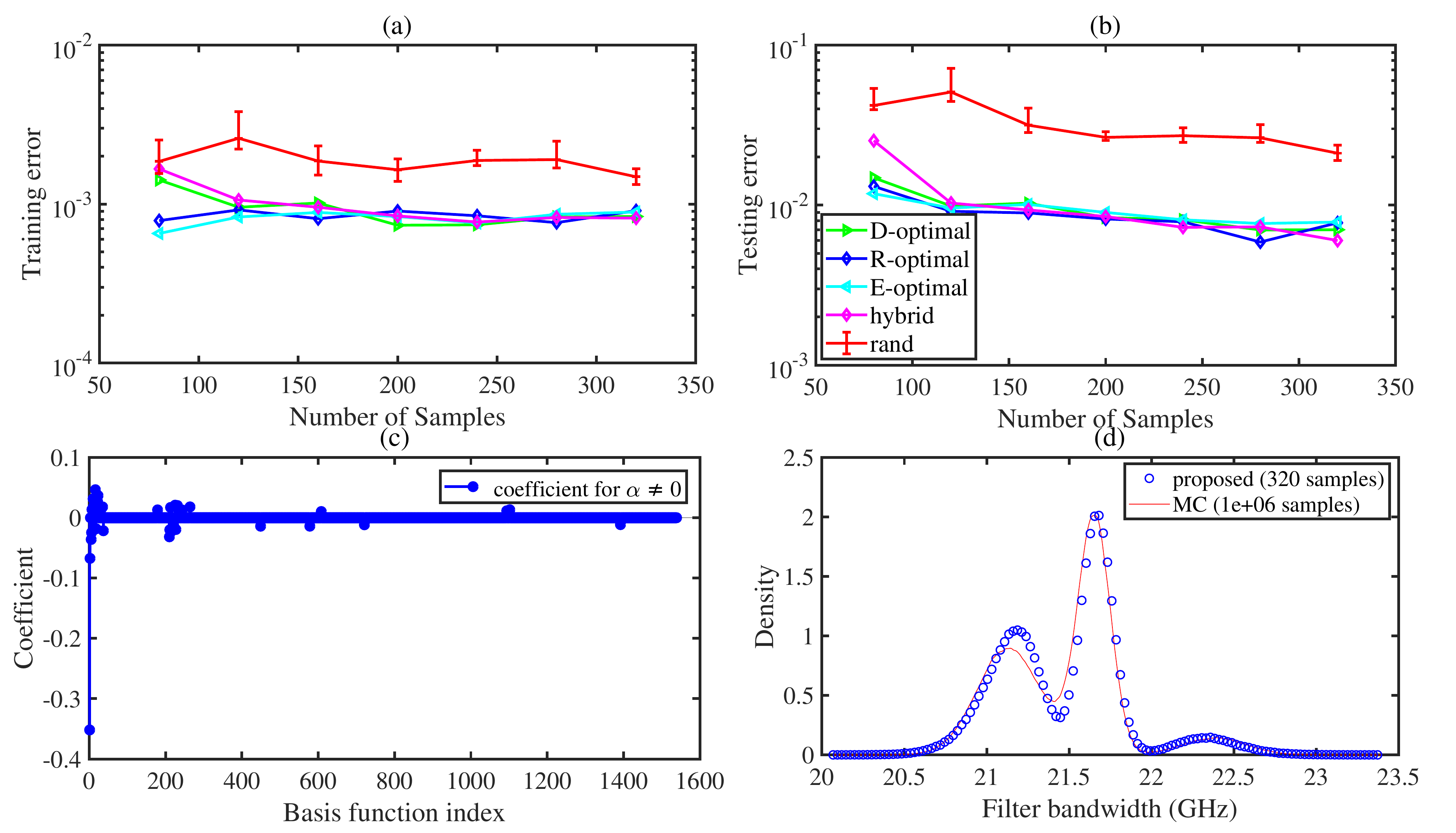}
\caption{Results for the {photonic} band-pass filter.
(a) training error; (b) testing error on 9000 new samples;
(c) calculated coefficients/weights of our proposed basis functions;
(d) probability density functions of the 3-dB bandwidth obtained with our proposed method and with Monte Carlo (MC), respectively.}
    \label{fig:res_PIC19}
\end{figure*}

\subsection{E-optimal Adaptive Sampling}

Both D-optimal and R-optimal methods solve an optimization problem and have to explore the entire sample sets. In contrast, we can exploit local information to pick the next sample. Inspired by the exploitation in Bayesian optimization, we propose to find the next sample in a  neighbourhood where the approximation error is large.
Specifically, we group the existing samples into $k$ clusters ${\cal U}^1,\ldots, {\cal U}^k$, and compute the average approximation error as
\begin{equation}
    \text{res}({\cal U}^i)=\text{mean}(\mat{\Phi}^i\mat{c}-\mat{y}^i),
\end{equation}
where $\mat{\Phi}^i\in\mathbb{R}^{|{\cal U}^i|\times n}$ contains the $|{\cal U}^i|$ rows of $\mat{\Phi}$ that are associated with all samples in ${\cal U}^i$.
Afterwards, we choose the next sample nearest to the $i$-th cluster center, where $i$ is the index of cluster with the maximal  residue, i.e.,  $i=\arg\max\text{res}({\cal U}^i)$. This approach is called E-optimal because it \textit{exploits} the samples in a neighbourhood.

\section{Numerical Results}
\label{sec:result}

We test our algorithms by a synthetic example and three real-world benchmarks, including a photonic band-pass filter, a 7-stage CMOS ring oscillator, and an array waveguide grating (AWG) with 41 waveguides.
For each example, we adaptively select a small number of samples from a pool of $1000$ candidate samples,
and we use $9000$ different samples for accuracy validation. We employ COSAMP~\cite{needell2009cosamp} to solve \reff{equ:fix_sparse}, {and stop it when the residue in the objective function satisfies $\|\mat{\Phi}\mat{c}-\mat y\|_2\le\epsilon$}.
We define the relative error as:
\begin{equation}
\label{equ:res}
\epsilon_r = \|\Phimat \cvec - \yvec\|_2 / \|\yvec\|_2.
\end{equation}
We refer $\epsilon_r$ as a training error if the samples are those used in our sparse solver, and as a testing (or prediction) error if an entirely new set of samples are used. The stopping criterion for Alg.~\ref{alg:AdaptiveSparseSolver} is either the maximal number of samples is attained or  the training error is small enough. We refer our methods as ``D-optimal", ``E-optimal", ``R-optimal", and ``hybrid" (combinations of all three methods), dependent on different sample selection criterion. We compare our methods with ``rand" approach that chooses all samples by Monte Carlo. For the ``rand" approach, we run the experiment 10 times using 10 sets of different samples, and report the mean values and variances of $\epsilon_r$. \ccf{The CPU time of obtaining each simulation sample highly depends on the specific design problem and on the hardware platform. In most cases, the simulation cost dominates the total cost of uncertainty quantification. Therefore, we compare the costs of different methods mainly based on their total numbers of simulation samples.}

\subsection{A Synthetic Example}

We firstly use a synthetic example to verify our theoretical results in Section~\ref{sec:sparsesolver}. This example contains $d=8$ non-Gaussian correlated random parameters $\vecpar$ and we approximate the stochastic solution $y(\vecpar)$ by our basis functions with a total order bounded by $p=3$.
The sparse coefficient $\mat{c}$ is given \textit{a \ccf{priori}}, and the output $y(\vecpar)$ has a closed-form  as
\begin{equation}
    y(\vecpar)=\mat{\Phi}(\vecpar)\mat c+\mat{e},
\end{equation}
where $\mat e$ is a  random simulation noise satisfying $\|\mat e\|_2=10^{-6}$.

\begin{figure}[t]
    \centering        \includegraphics[width=2.0in]{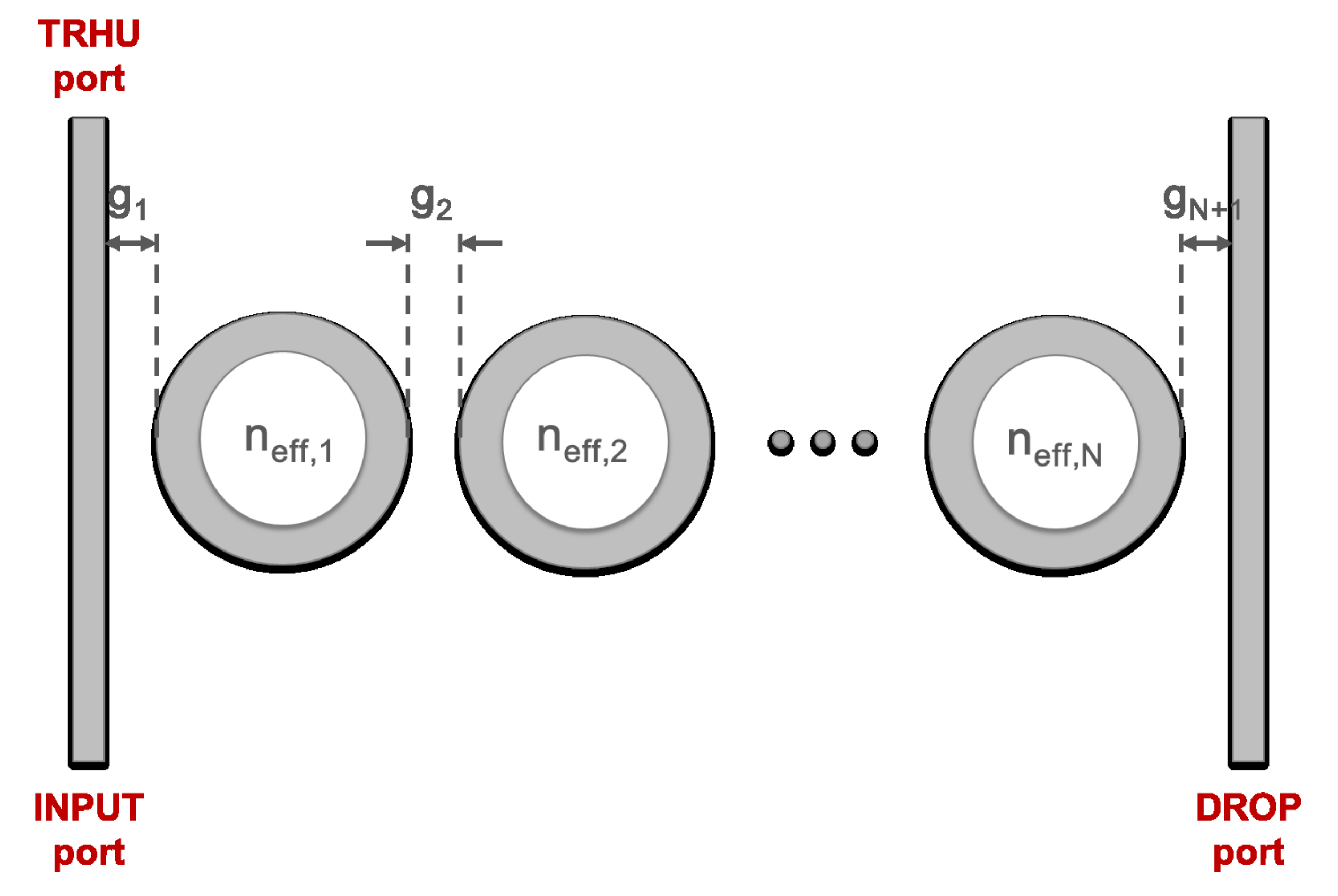}
\caption{A band-pass filter with 9 micro-ring resonators.}
    \label{fig:photonics}
\end{figure}

\begin{figure*}[t]
    \centering
        \includegraphics[width=5.0in,height=3.1in]{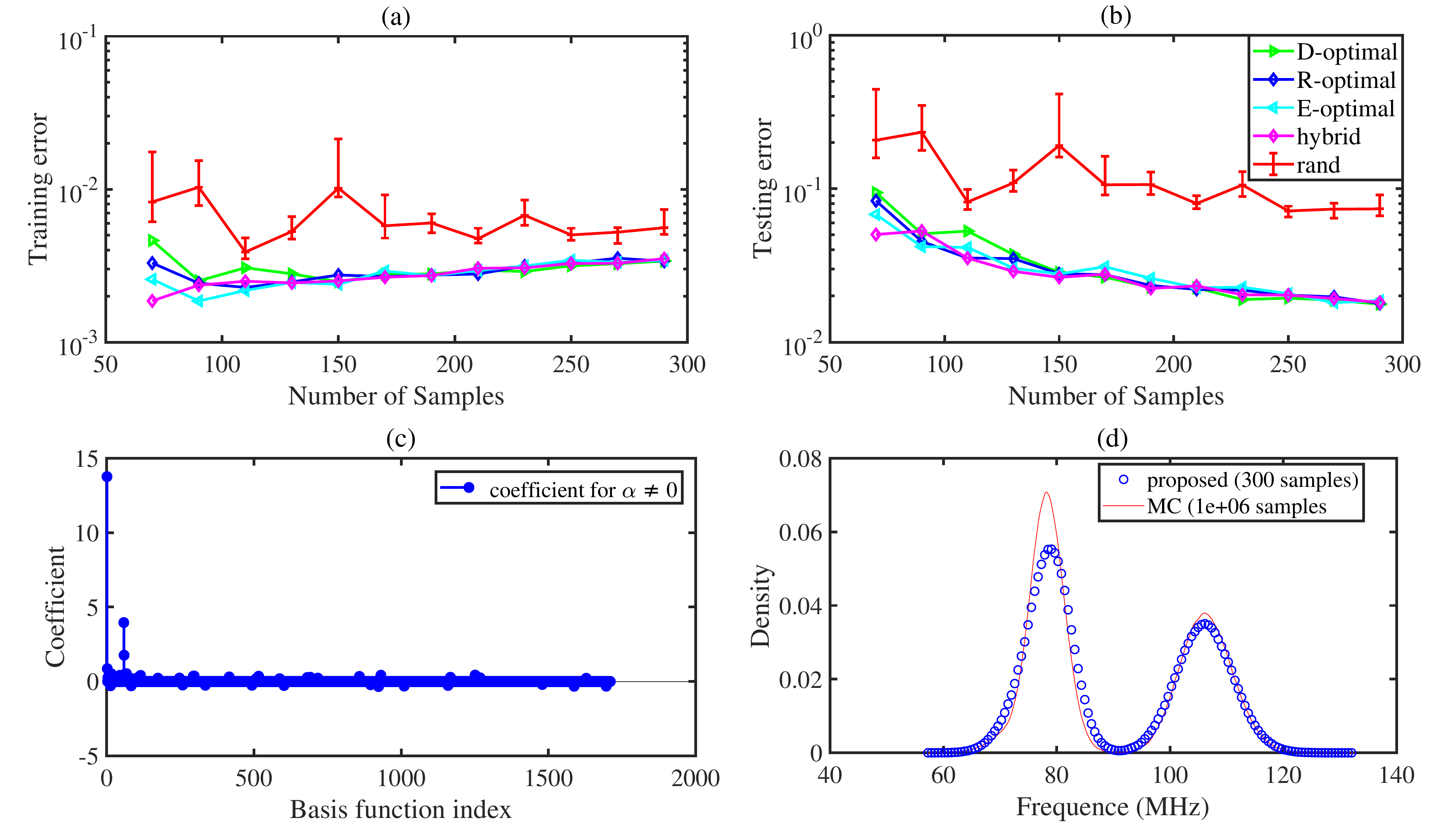}
\caption{Results for the CMOS ring oscillator. (a) training error; (b) testing error on 9000 new samples;
(c) calculated coefficients/weights of our proposed basis functions;
(d) probability density functions of the oscillator frequency obtained by our method with 300 training samples and by Monte Carlo, respectively.}
    \label{fig:res_ring57}
\end{figure*}

In order to verify Theorems \ref{thm:accuracy} and \ref{thm:approxerr}, we generate $m=200$ random samples and  approximate the sparse coefficients via an $\ell_0$-minimization  \reff{equ:fix_sparse}.   {We stop the algorithm when $\|\mat{\Phi c-y}\|_2\le\epsilon$}.
Fig.~\ref{fig:synethic_err_c}~(a) shows that  when the numerical error $\epsilon$ is too large, the error will always be dominated by $\epsilon$. Otherwise, when $\epsilon$ is small enough, the prescribed sparsity $s$ will dominate the error. This is consistent with Theorem~\ref{thm:accuracy}. Fig.~\ref{fig:synethic_err_c}~(b) confirms Theorem~\ref{thm:approxerr}: when $y(\vecpar)$ and the polynomial order $p$ are fixed (hence $\|y(\vecpar)-y_p(\vecpar)\|_2$ is fixed), the overall error is entirely dependent on the coefficient error.

\begin{figure}[t]
    \centering
      \includegraphics[width=2.1in]{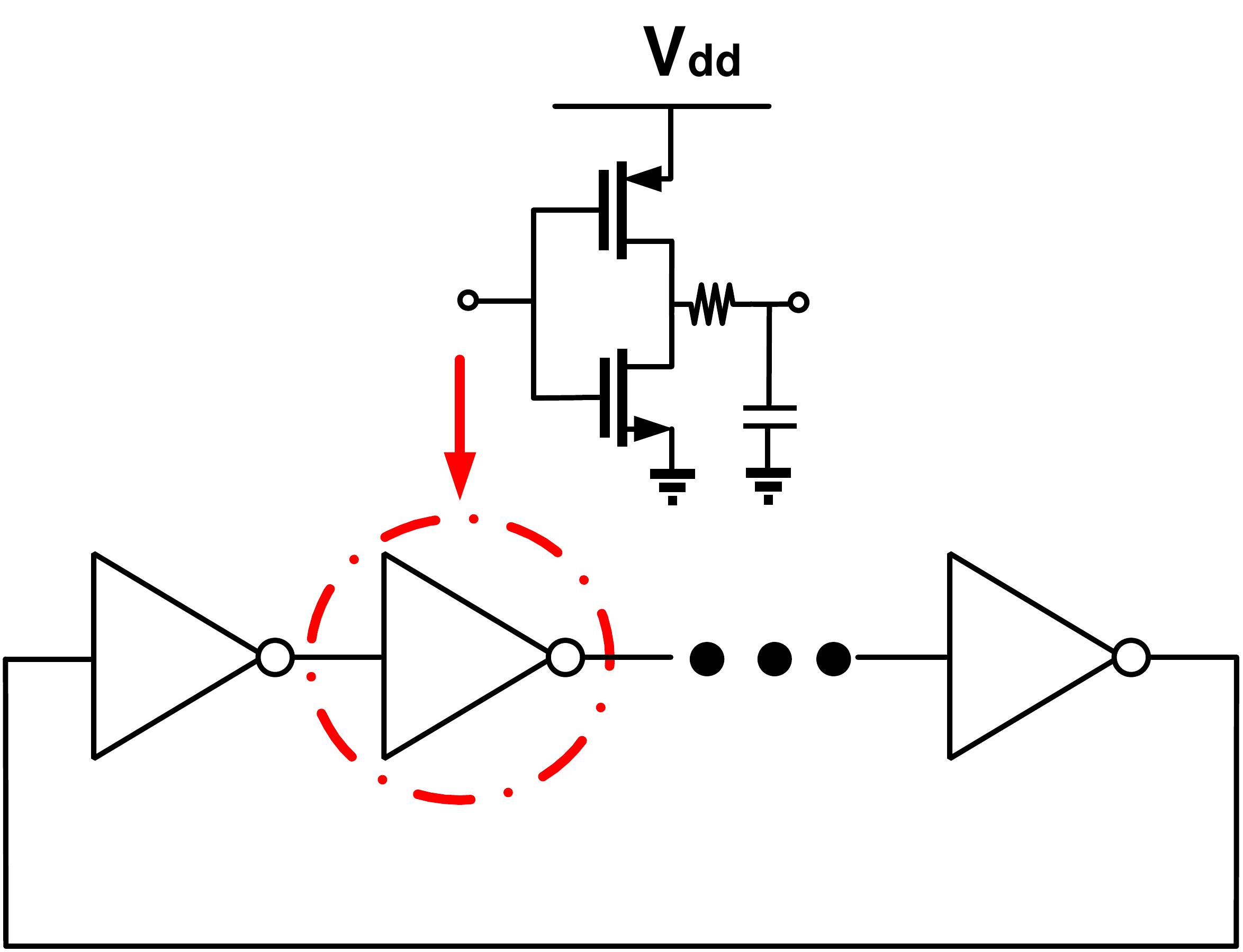}
\caption{Schematic of a CMOS ring oscillator. }
    \label{fig:ring}
\end{figure}

\textit{Remark.} Fig.~\ref{fig:synethic_err_c} shows that a large $s$ leads to smaller errors when $\epsilon$ is small enough. However, we cannot set the sparsity $s$ to be too large, because a larger $s$ requires more samples to achieve the RIP condition. Therefore, in the following experiments, we set $s$ as the largest integer below $\frac{m}{3}$.

\subsection{Photonic Band-pass Filter (19 Parameters)}

Now we consider the photonic band-pass filter in Fig.~\ref{fig:photonics}. This photonic IC has 9 micro-ring resonators, and it was originally designed to have a 3-dB bandwidth of 20 GHz, a 400-GHz free spectral range, and a 1.55-$\mu$m operation wavelength. A total of $19$ random parameters are used to describe the variations of the effective phase index ($n_{\text{neff}}$) of each ring, as well as the gap ($g$) between adjacent rings and between the first/last ring and the bus waveguides. These non-Gaussian correlated random parameters are described by a Gaussian  mixture  with three components. 

\begin{table}
\caption{Accuracy comparison on the photonic band-pass filter. The underscores indicate precision.}
\centering
\begin{tabular}{c|c|c|c|c}
\hline
method& Proposed &\multicolumn{3}{c}{Monte Carlo} \\ \hline
\# samples & 320 & $10^2$& $10^4$& $10^6$\\ \hline
mean (GHz)& 21.4773 & 2\underline{1}.5297& 21. \underline{4}867& 21.4\underline{7}82\\  \hline
{std (GHz)} &  0.3884 &0.4131 &   0.3767  &  0.3808  \\ \hline
\end{tabular}
\label{tab:pic}
\end{table}

We approximate the 3-dB bandwidth $f_{3 {\rm dB}}$ at the DROP port using our basis functions with the total order bounded by $p=3$. {It takes 350.59 seconds to generate the   1540 basis functions}.
We verify D-optimal, R-optimal, E-optimal methods, and their combinations (denoted as ``hybrid"). Fig.~\ref{fig:res_PIC19}~(b) clearly shows that all four adaptive sampling methods lead to significantly lower testing (i.e., prediction) errors because they choose more informative samples. Finally, we use 320 samples to assemble a linear system and solve it by an $\ell_0$ minimization, and obtain the sparse coefficients of our basis functions in Fig.~\ref{fig:res_PIC19}~(c). Although a third-order expansion involves more than $1000$ basis functions, only a few dozens are important.
Fig.~\ref{fig:res_PIC19}~(d) shows the predicted probability density function of the filter's 3-dB bandwidth, and it matches the result from Monte Carlo very well. More importantly, it is clear that our algorithm can capture accurately the multiple peaks in the output density function, and these peaks can be hardly predicted using existing stochastic spectral methods.

In order to demonstrate the effectiveness of our stochastic model, we compare the computed mean values and standard variations of $f_{3 {\rm dB}}$ from our methods with that of Monte Carlo in Table~\ref{tab:pic}. Our method provides a closed-form expression for the mean value. Monte Carlo method converges very slowly and requires $3125\times$ more simulation samples to achieve the similar level of accuracy (with 2 accurate fractional digits).

\subsection{CMOS Ring Oscillator (57 Parameters)}

We continue to consider the 7-stage CMOS ring oscillator in Fig.~\ref{fig:ring}. This circuit has $57$ random parameters describing the variations of threshold voltages, gate-oxide thickness, and effective gate length/width. We use a three-component Gaussian mixture model to describe the strong non-Gaussian correlations of threshold voltages, gate oxide thickness, gate lengths and widths.

We employ a 2nd-order expansion to model the oscillator frequency. {Generating the 1711 basis functions takes 1657 seconds.} The simulation samples are obtained by calling a periodic steady-state simulator repeatedly. The detailed results are shown in Fig.~\ref{fig:res_ring57}. Our adaptive sparse solver produces a sparse and highly accurate stochastic solution with better prediction behaviors than the standard compressed sensing does. The proposed basis functions can well capture the multiple peaks of the output probability density function caused by the strong non-Gaussian correlation.

Table~\ref{tab:ring} compares our method with Monte Carlo. Our method takes about $3333\times$ less samples than Monte Carlo to achieve a precision of one fractional digit for the mean value.

\begin{figure*}[t]
    \centering
    \includegraphics[width=4.6in]{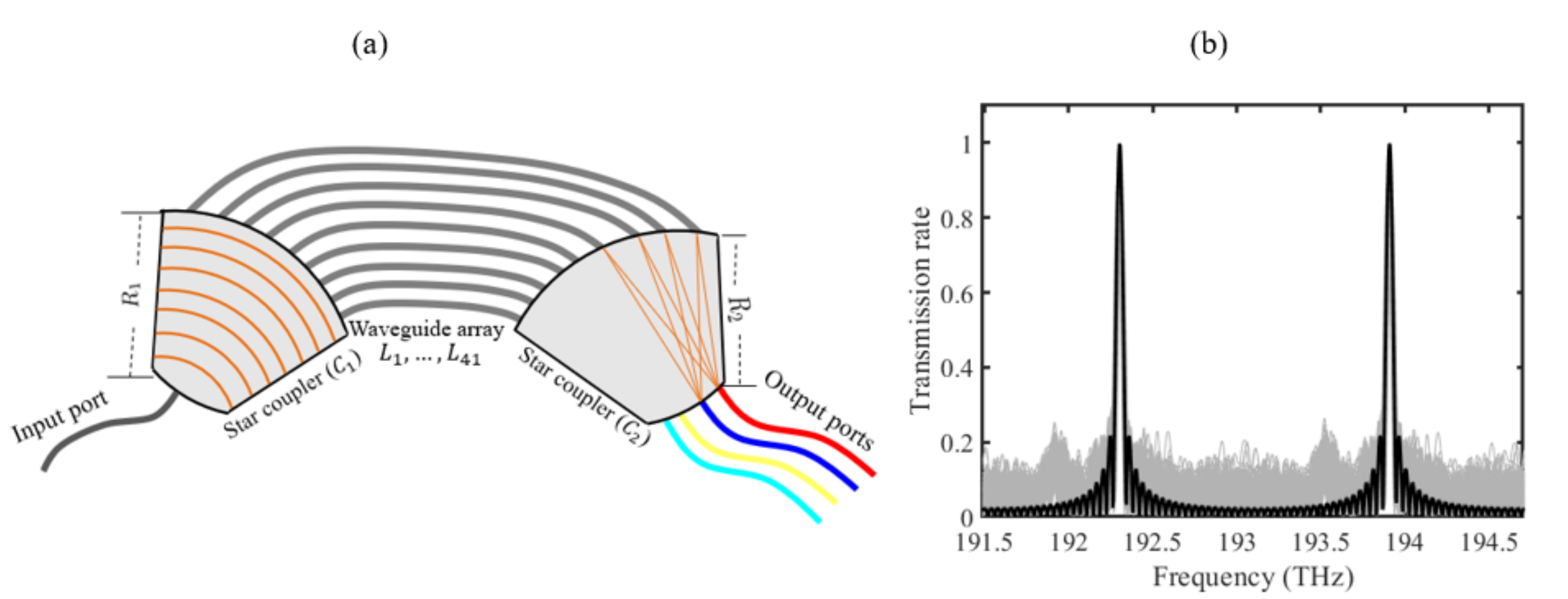}
\caption{(a) An AWG with 41 waveguide arrays; (b) The transmission rates from the input to output Port 1. The black curve shows the nominal result without any uncertainties, and the grey curves show the effects caused by the fabrication uncertainties of radius $R_1$, $R_2$ and waveguide lengths $L_1,\ldots,L_{41}$.}
    \label{fig:AWG}
\end{figure*}

\begin{figure*}[t]
     \centering
     \includegraphics[width=5.0in,height=3.1in]{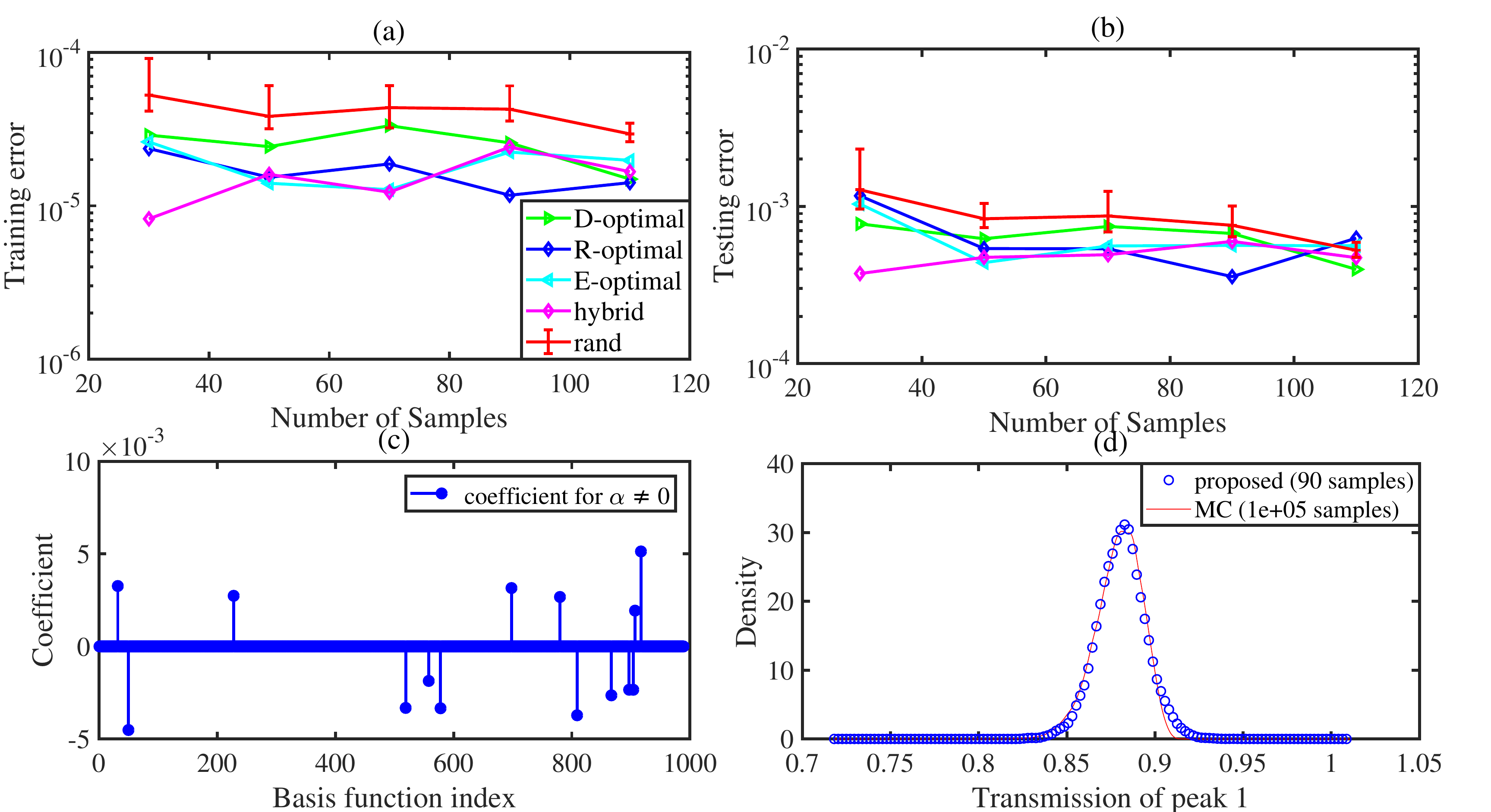}
     \caption{Numerical results of the AWG example. (a) Training error; (b) testing error on $10^{5}$ new samples; (c) computed coefficients/weights of our basis functions; (d) probability density function of transmission of peak one with our proposed method  and Monte Carlo, respectively.}
     \label{fig:AWG_result}
 \end{figure*}

\begin{table}[t]
\caption{Accuracy comparison on the CMOS ring oscillator. The underscores indicate precision.}
\centering
\begin{tabular}{c|c|c|c|c}
\hline
method& Proposed &\multicolumn{3}{c}{Monte Carlo} \\ \hline
\# samples & 300 & $10^2$& $10^4$& $10^6$\\ \hline
mean (MHz) &90.5797 & 89.7795& 9\underline{0}.4945& 90.\underline{5}253\\ \hline
{std (MHz)}& 14.6068 & 14.4512  & 14.6975 &  14.7400     \\
\hline
\end{tabular}
\label{tab:ring}
\end{table}

\subsection{Array Waveguide Grating (AWG, 43 Parameters)}

Finally, we investigate an arrayed waveguide grating (AWG)~\cite{zhang2018verilog}. The AWG is essential for wavelength division multiplexing in photonic systems. We use an AWG with 41 waveguide arrays and two-star couplers, as shown in Fig. \ref{fig:AWG}~(a). In the nominal design, the radius of each star coupler is $R_1=R_2= 2.985$ mm, and the waveguide lengths $L_1,\ldots,L_{41}$ range from $46$ $ \mu $m to $1.9$ $m$m. We use a Gaussian  mixture distribution with two components to describe the uncertainties in the waveguide lengths, and a Gamma distribution to formulate the uncertainties in each star coupler. The resulting transmission with uncertainties is shown in Fig.~\ref{fig:AWG}~(b).

We approximate the transmission rate of peak 1 at the output 1 by our proposed basis functions with a total order $p=2$. {It takes 227.39 seconds to generate 990 basis functions.} The numerical results are presented in Fig.~\ref{fig:AWG_result}. Similar to the previous examples, our adaptive sparse solver produces a sparse and highly accurate stochastic solution with better prediction accuracy than using random samples. Table~\ref{tab:AWG} compares our method with Monte Carlo. Our method consumes about $1111\times$ less sample points than Monte Carlo to get two exact fractional digits for the mean value.

\begin{table}
\caption{Accuracy comparison on the AWG by adaptive sampling. The underscores indicate precision.}
\centering
\begin{tabular}{c|c|c|c}
\hline
method& Proposed &\multicolumn{2}{c}{Monte Carlo} \\ \hline
\# samples & 90 & $10^3$& $10^5$\\ \hline
mean &0.8739&  0.\underline{8}802 &    0.8\underline{7}92    \\\hline
{std } &0.0120 &0.0134  &  0.0135   \\\hline

\end{tabular}
\label{tab:AWG}
\end{table}

\section{Conclusion}
\label{sec:conclusion}

This paper has presented a set of theoretical and numerical results for high-dimensional uncertainty quantification with non-Gaussian correlated process variations. We have proposed a set of basis functions for non-Gaussian correlated cases, and have provided a functional tensor-train method for their high-dimensional implementation. Theoretical results on the expressivity of our basis function are presented.
In order to reduce the computational time of analyzing process variations, we have justified the theoretical foundations (i.e., theoretical conditions and numerical errors) of compressed sensing in our problem setting. We have also proposed several adaptive sampling techniques to improve the performance of compressed sensing. Our approach has been verified with a synthetic example and three electronic and photonic ICs with up to $57$ random parameters. On these benchmarks, our method has achieved high accuracy in predicting the multi-peak output probability density functions and in estimating the output mean value. Our method has consumed  $1111\times$ to $3333\times$ less samples than Monte Carlo to achieve a similar level of accuracy.


\appendices
 
\section{Proof of Lemma \ref{lem:complete}}
\label{app:complete}

We show the completeness of our basis function in $\ten{S}_p$ via two steps.  Firstly, it follows from the definition of polynomials that the monomials \eqref{equ:monomials} are complete basis functions for  $\ten{S}_p$: any $y(\vecpar)\in\ten{S}_p$ can be written as  $y(\vecpar)=\mat{c}_0^T\mat{b}(\vecpar)$.
Secondly, our basis function is an equivalent linear transformation from the monomials $\mat{b}(\vecpar)$. Consequently,
\begin{equation*}
y(\vecpar)=\mat{c}_0^T\mat{b}(\vecpar)=\mat{c}_0^T\mat{L}\mat{\Phi}(\vecpar):=\mat{c}^T\mat{\Phi}(\vecpar).
\end{equation*}
This shows that our proposed basis function is complete (any function $y(\vecpar)\in\ten{S}_p$ can be expressed by a linear transformation of our proposed basis function).

\section{Proof of Lemma \ref{lem:convergence}} \label{app:convergence}
The statements (i) and (ii) hold if $\ten{S}_p$ is dense in $L^2(\vecpar,\rho(\vecpar))$.
According to Theorem 3 of \cite{petersen1983relation},  a sufficient condition would be the following: there exists $q>2$ such that the 1-D polynomials are dense in  $L^q(\xi_i,\rho_i(\xi_i)), \forall\, i=1,\ldots,d$. Here $\rho_i(\xi_i)$ is the  marginal distribution of $\xi_i$.

Consider the following two cases.
(i),   the marginal distribution  $\rho_i(\xi_i)$ is defined on a compact domain.  Then the 1-D polynomials are dense in  $L^q(\xi_i,\rho_i(\xi_i))$ under the Weierstrass Approximation Theorem \cite{perez2008survey}.
(ii),   $\rho_i(\xi_i)$ is  defined on a non-compact domain.
In this case, the 1-D polynomials are dense in  $L^q(\xi_i,\rho_i(\xi_i))$ under the condition that  the random variables are exponentially integrable  \cite{petersen1983relation}.  Namely, there exits a constant $a\ge0$ such that
\begin{equation}\label{equ:exponinteg}
\mathbb{E}[\exp(a|\xi_i|)]=\int_{\mathbb{R}}\exp(a|\xi_i|)\rho_i(\xi_i)d\xi_i <\infty, \forall\, i.
\end{equation}
In both two cases, the 1-D polynomials are dense in  $L^q(\xi_i,\rho_i(\xi_i)), \forall\, i=1,\ldots,d$,   hence the multidimensional polynomials are dense in $L^2(\vecpar, \rho(\vecpar))$ \cite{petersen1983relation}.

\textsl{Remark.} Inequality \reff{equ:exponinteg} holds  for many well-known distributions, such as   normal distribution,  Gaussian mixture distribution, and Gamma distribution.
 
\section{Proof of Theorem \ref{thm:MBound}}
\label{app:Mbound}

A sufficient condition to achieve the ($s,\kappa_s$)-RIP condition is {if} the following inequality holds
\begin{equation}\label{equ:sufRIP}
\|\frac1m \mat{\Phi}_s^T\mat{\Phi}_s - \mat{I}_s\|_F \le \kappa_s
\end{equation}
for any $\mat{\Phi}_s$  constructed by arbitrary $s$ columns of   $\mat{\Phi}$.
Equation \eqref{equ:sufRIP} can be derived if each element satisfies
\begin{equation}\label{equ:x}
|\frac1m \sum_kX_k^{ij}-\delta_{ij}|\le \frac{\kappa_s}{s}.
\end{equation}
Here $\delta_{ij}=1$ if $i=j$, and  $\delta_{ij}=0$ otherwise.
It follows from the concentration bounds of sub-Gaussian random variables \cite{buldygin1980sub} that for any $t\ge0$ there is
\begin{equation}
P\left(\left|\frac1m \sum_{k=1}^m(X_k^{ij}-\delta_{ij})\right|\ge \frac{t}{m}\right) \le 2\exp(-\frac{t^2}{2m\sigma^2}).
\end{equation}
 Substituting $t=m\frac{\kappa_s}{s}$ and $2\exp(-\frac{t^2}{2m\sigma^2})\le \eta$ into the above equation, we have that \eqref{equ:x} holds with a probability $\geq 1-\eta$ if $2\exp(-\frac{t^2}{2m\sigma^2})=2\exp(-\frac{m\kappa_s^2}{2s^2\sigma^2})\le \eta$ (i.e., $m\ge 2\log\left(\frac{2}{\eta}\right)\frac{s^2\sigma^2}{\kappa_s^2}$).

\section{Proof of Theorem \ref{thm:accuracy}}
\label{append:L0error}
Denote vector $\mat c$ as the exact unknown coefficients, vector $\mat c_s$ as the $s$-sparsity approximation of $\mat c$, and vector $\mat c^*$ as the solution from our $\ell_0$-minimization solver. The error of $\mat{c}^*$ satisfies
\begin{align*}
    \|\mat{c}-\mat{c}^*\|_2\le&\|\mat{c}-\mat{c}_s\|_2 + \|\mat{c}_s-\mat{c}^*\|_2\\
    \le&\|\mat{c}-\mat{c}_s\|_2 + \frac{1}{m(1-\kappa_{2s})}\|\mat{\Phi}\mat{c}_s-\mat{\Phi}\mat{c}^*\|_2,
\end{align*}
where
\begin{align*}
    \|\mat{\Phi}\mat{c}_s-\mat{\Phi}\mat{c}^*\|_2\le&\|\mat{\Phi}\mat{c}_s-\mat{\Phi}\mat{c}\|_2 + \|\mat{\Phi}\mat{c}-\mat{\Phi}\mat{c}^*\|_2\\
    \le&\|\mat{\Phi}\mat{c}_s-\mat{\Phi}\mat{c}\|_2 + \|\mat y-\mat{\Phi}\mat{c}^*\|_2 +\|\mat{e}\|_2\\
    \le &\|\mat{\Phi}\mat{c}_s-\mat{\Phi}\mat{c}\|_2 +{\epsilon} +\|\mat{e}\|_2.
\end{align*}
By Proposition 3.5 in \cite{needell2009cosamp}, it holds that
\begin{align*}
  \|\mat{\Phi}(\mat{c}_s-\mat{c})\|_2  \le& \sqrt{1+\kappa_{2s}}(\|\mat{c}_s-\mat{c}\|_2+\frac{1}{\sqrt{2s}}\|\mat{c}_s-\mat{c}\|_1)\\
  \le & \frac{1.7071 \sqrt{1+\kappa_{2s}}}{\sqrt{s}}\|\mat{c}_s-\mat{c}\|_1.
\end{align*}
Combing the above equations together, we have
\begin{align}
   \nonumber &\|\mat{c}-\mat{c}^*\|_2\\
 \nonumber   \le&\|\mat{c}-\mat{c}_s\|_1 +\left(\frac{1.7071 \sqrt{1+\kappa_{2s}}}{ m(1-\kappa_{2s}) \sqrt{s}}\|\mat{c}_s-\mat{c}\|_1 +\epsilon +\|\mat{e}\|_2\right)\\
   =& \alpha_0\|\mat{c}_s-\mat{c}\|_1+\alpha_1(\epsilon+\|\mat{e}\|_2),\label{equ:coefferr}
\end{align}
where $\alpha_0=1+\frac{1.7071 \sqrt{1+\kappa_{2s}}}{m(1-\kappa_{2s})\sqrt{s}}$, $\alpha_1=\frac{1}{m(1-\kappa_{2s})}$ are constants.

\section{Proof of Theorem \ref{thm:approxerr}}
\label{app:approxerr}
Denote $y(\vecpar)$ as the unknown quantity of interest,  $y_p(\vecpar)=\sum_{|\basisInd|=0}^pc_{\basisInd}\multiGPC_{\basisInd}(\vecpar)$ as the projection of $y(\vecpar)$ onto the $p$-th order polynomial space $\ten{S}_p$, and $y^*(\vecpar)=\sum_{|\basisInd|=0}^pc_{\basisInd}^*\multiGPC_{\basisInd}(\vecpar)$ as the model from our  numerical framework, then we have
 \begin{equation}
     \|y(\vecpar)-y^*(\vecpar)\|_2\le \|y(\vecpar)-y_p(\vecpar)\|_2+ \|y_p(\vecpar)-y^*(\vecpar)\|_2.
 \end{equation}
The first term is the distance of $y(\vecpar)$ to $\ten{S}_p$, which can be very small if $p$ is large enough. The second term is due to the error caused by a compressed sensing solver:
\begin{align}
  \nonumber  &\|y_p(\vecpar)-y^*(\vecpar)\|_2={\sqrt{\mathbb{E}[(y_p(\vecpar)-y^*(\vecpar))^2]}}\\
 \nonumber =&\sqrt{\mathbb{E}[(\sum_{|\alpha|=0}^p (c_{\basisInd}-c_{\basisInd}^*)\multiGPC_{\basisInd}(\vecpar))^2]}= \sqrt{\sum_{|\alpha|=0}^p (c_{\basisInd}-c_{\basisInd}^*)^2}\\
  \nonumber =&\|\mat{c}-\mat{c}^*\|_2,
\end{align}
where the third equality is due to the orthonormal property of our basis functions \ccf{and the last equality is derived from the definition of the $\ell_2$-norm in the Euclidean space}.

\bibliographystyle{IEEEtran}
\bibliography{RefList}

\begin{thebibliography}{10}
\providecommand{\url}[1]{#1}
\csname url@samestyle\endcsname
\providecommand{\newblock}{\relax}
\providecommand{\bibinfo}[2]{#2}
\providecommand{\BIBentrySTDinterwordspacing}{\spaceskip=0pt\relax}
\providecommand{\BIBentryALTinterwordstretchfactor}{4}
\providecommand{\BIBentryALTinterwordspacing}{\spaceskip=\fontdimen2\font plus
\BIBentryALTinterwordstretchfactor\fontdimen3\font minus
  \fontdimen4\font\relax}
\providecommand{\BIBforeignlanguage}[2]{{%
\expandafter\ifx\csname l@#1\endcsname\relax
\typeout{** WARNING: IEEEtran.bst: No hyphenation pattern has been}%
\typeout{** loaded for the language `#1'. Using the pattern for}%
\typeout{** the default language instead.}%
\else
\language=\csname l@#1\endcsname
\fi
#2}}
\providecommand{\BIBdecl}{\relax}
\BIBdecl

\bibitem{cui2018uncertainty}
C.~Cui and Z.~Zhang, ``Uncertainty quantification of electronic and photonic
  {ICs} with non-{Gaussian} correlated process variations,'' in \emph{Proc.
  Intl. Conf. Computer-Aided Design}, 2018, pp. 1--8.

\bibitem{variation2008}
D.~S. Boning, K.~Balakrishnan, H.~Cai \emph{et~al.}, ``Variation,'' \emph{IEEE
  Trans. Semiconductor Manufacturing}, vol.~21, no.~1, pp. 63--71, 2008.

\bibitem{agarwal2009stochastic}
N.~Agarwal and N.~R. Aluru, ``Stochastic analysis of electrostatic mems
  subjected to parameter variations,'' \emph{Journal of Microelectromechanical
  Systems}, vol.~18, no.~6, pp. 1454--1468, 2009.

\bibitem{zortman2010silicon}
W.~A. Zortman, D.~C. Trotter, and M.~R. Watts, ``Silicon photonics
  manufacturing,'' \emph{Optics express}, vol.~18, no.~23, pp.
  23\,598--23\,607, 2010.

\bibitem{MCintro}
S.~Weinzierl, ``Introduction to {Monte Carlo} methods,'' {NIKHEF}, Theory
  Group, The Netherlands, Tech. Rep. NIKHEF-00-012, 2000.

\bibitem{ghanem1991stochastic}
R.~G. Ghanem and P.~D. Spanos, ``Stochastic finite element method: Response
  statistics,'' in \emph{Stochastic finite elements: a spectral
  approach}.\hskip 1em plus 0.5em minus 0.4em\relax Springer, 1991, pp.
  101--119.

\bibitem{zzhang:tcad2013}
Z.~Zhang, T.~A. El-Moselhy, I.~A.~M. Elfadel, and L.~Daniel, ``Stochastic
  testing method for transistor-level uncertainty quantification based on
  generalized polynomial chaos,'' \emph{IEEE Trans. CAD Integrated Circuits
  Syst.}, vol.~32, no.~10, Oct. 2013.

\bibitem{xiu2005high}
D.~Xiu and J.~S. Hesthaven, ``High-order collocation methods for differential
  equations with random inputs,'' \emph{SIAM J. Sci. Comp.}, vol.~27, no.~3,
  pp. 1118--1139, 2005.

\bibitem{Tarek_DAC:08}
T.~Moselhy and L.~Daniel, ``Stochastic integral equation solver for efficient
  variation aware interconnect extraction,'' in \emph{Proc. Design Auto.
  Conf.}, Jun. 2008, pp. 415--420.

\bibitem{Wang:2004}
J.~Wang, P.~Ghanta, and S.~Vrudhula, ``Stochastic analysis of interconnect
  performance in the presence of process variations,'' in \emph{Proc. DAC},
  2004, pp. 880--886.

\bibitem{Shen2010}
R.~Shen, S.~X.-D. Tan, J.~Cui, W.~Yu, Y.~Cai, and G.~Chen, ``Variational
  capacitance extraction and modeling based on orthogonal polynomial method,''
  \emph{IEEE Trans. VLSI}, vol.~18, no.~11, pp. 1556 --1565, Nov. 2010.

\bibitem{cmpt2012}
D.~V. Ginste, D.~D. Zutter, D.~Deschrijver, T.~Dhaene, P.~Manfredi, and
  F.~Canavero, ``Stochastic modeling-based variability analysis of on-chip
  interconnects,'' \emph{IEEE Trans. Comp, Pack. Manufact. Tech.}, vol.~2,
  no.~7, pp. 1182--1192, Jul. 2012.

\bibitem{chen2014optimal}
X.~Chen, J.~S. Ochoa, J.~E. Schutt-Ain{\'e}, and A.~C. Cangellaris, ``Optimal
  relaxation of {I/O} electrical requirements under packaging uncertainty by
  stochastic methods,'' in \emph{Electronic Components and Technology
  Conference}, 2014, pp. 717--722.

\bibitem{pham2014decoupled}
T.-A. Pham, E.~Gad, M.~S. Nakhla, and R.~Achar, ``Decoupled polynomial chaos
  and its applications to statistical analysis of high-speed interconnects,''
  \emph{IEEE Trans. Components, Packaging and Manufacturing Technology},
  vol.~4, no.~10, pp. 1634--1647, 2014.

\bibitem{Strunz:2008}
K.~Strunz and Q.~Su, ``Stochastic formulation of {SPICE}-type electronic
  circuit simulation with polynomial chaos,'' \emph{ACM Trans. Modeling and
  Computer Simulation}, vol.~18, no.~4, pp. 15:1--15:23, Sep 2008.

\bibitem{Tao:2007}
J.~Tao, X.~Zeng, W.~Cai, Y.~Su, D.~Zhou, and C.~Chiang, ``Stochastic
  sparse-grid collocation algorithm ({SSCA}) for periodic steady-state analysis
  of nonlinear system with process variations,'' in \emph{Porc. Asia and South
  Pacific Design Automation Conference}, 2007, pp. 474--479.

\bibitem{spina2012variability}
D.~Spina, F.~Ferranti, T.~Dhaene, L.~Knockaert, G.~Antonini, and D.~V. Ginste,
  ``Variability analysis of multiport systems via polynomial-chaos expansion,''
  \emph{IEEE Trans. Microwave Theory and Techniques}, vol.~60, no.~8, pp.
  2329--2338, 2012.

\bibitem{manfredi2014stochastic}
P.~Manfredi, D.~V. Ginste, D.~De~Zutter, and F.~G. Canavero, ``Stochastic
  modeling of nonlinear circuits via spice-compatible spectral equivalents,''
  \emph{IEEE Transactions on Circuits and Systems I: Regular Papers}, vol.~61,
  no.~7, pp. 2057--2065, 2014.

\bibitem{Rufuie2014}
M.~Rufuie, E.~Gad, M.~Nakhla, R.~Achar, and M.~Farhan, ``Fast variability
  analysis of general nonlinear circuits using decoupled polynomial chaos,'' in
  \emph{Workshop Signal and Power Integrity}, May 2014, pp. 1--4.

\bibitem{ahadi2016sparse}
M.~Ahadi and S.~Roy, ``Sparse linear regression (spliner) approach for
  efficient multidimensional uncertainty quantification of high-speed
  circuits,'' \emph{IEEE Trans. CAD of Integr. Circuits Syst.}, vol.~35,
  no.~10, pp. 1640--1652, 2016.

\bibitem{zzhang_cicc2014}
Z.~Zhang, X.~Yang, G.~Marucci, P.~Maffezzoni, I.~M. Elfadel, G.~Karniadakis,
  and L.~Daniel, ``Stochastic testing simulator for integrated circuits and
  {MEMS}: Hierarchical and sparse techniques,'' in \emph{Proc. Custom
  Integrated Circuit Conf.}\hskip 1em plus 0.5em minus 0.4em\relax CA, Sept.
  2014, pp. 1--8.

\bibitem{twweng:optsEx}
T.-W. Weng, Z.~Zhang, Z.~Su, Y.~Marzouk, A.~Melloni, and L.~Daniel,
  ``Uncertainty quantification of silicon photonic devices with correlated and
  non-{Gaussian} random parameters,'' \emph{Optics Express}, vol.~23, no.~4,
  pp. 4242 -- 4254, Feb 2015.

\bibitem{waqas2018polynomial}
A.~Waqas, D.~Melati, P.~Manfredi, F.~Grassi, and A.~Melloni, ``A
  polynomial-chaos-expansion-based building block approach for stochastic
  analysis of photonic circuits,'' in \emph{Physics and Simulation of
  Optoelectronic Devices XXVI}, vol. 10526, 2018, p. 1052617.

\bibitem{melati2015statistical}
D.~Melati, E.~Lovati, and A.~Melloni, ``Statistical process design kits:
  analysis of fabrication tolerances in integrated photonic circuits,'' in
  \emph{Integrated Photonics Research, Silicon and Nanophotonics}, 2015, pp.
  IT4A--5.

\bibitem{he2019iccad}
Z.~He, W.~Cui, C.~Cui, T.~Sherwood, and Z.~Zhang, ``Efficient uncertainty
  modeling for system design via mixed integer programming,'' in \emph{Proc.
  Intl. Conf. Computer-Aided Design}, Nov. 2019.

\bibitem{gPC2002}
D.~Xiu and G.~E. Karniadakis, ``The {Wiener-Askey} polynomial chaos for
  stochastic differential equations,'' \emph{SIAM J. Sci. Comp.}, vol.~24,
  no.~2, pp. 619--644, Feb 2002.

\bibitem{wold1987principal}
S.~Wold, K.~Esbensen, and P.~Geladi, ``Principal component analysis,''
  \emph{Chemometrics and intelligent laboratory systems}, vol.~2, no. 1-3, pp.
  37--52, 1987.

\bibitem{singh2006statistical}
J.~Singh and S.~Sapatnekar, ``Statistical timing analysis with correlated
  non-{Gaussian} parameters using independent component analysis,'' in
  \emph{Proc. DAC}, 2006, pp. 155--160.

\bibitem{soize2004physical}
C.~Soize and R.~Ghanem, ``Physical systems with random uncertainties: chaos
  representations with arbitrary probability measure,'' \emph{SIAM Journal on
  Scientific Computing}, vol.~26, no.~2, pp. 395--410, 2004.

\bibitem{cui2018stochastic_epeps}
C.~Cui, M.~Gershman, and Z.~Zhang, ``Stochastic collocation with non-{Gaussian}
  correlated parameters via a new quadrature rule,'' in \emph{Proc. IEEE Conf.
  Electrical Performance of Electronic Packaging and Syst.}, 2018, pp. 57--59.

\bibitem{cui2018stochastic}
C.~Cui and Z.~Zhang, ``Stochastic collocation with non-{Gaussian} correlated
  process variations: Theory, algorithms and applications,'' \emph{IEEE Trans.
  Components, Packaging and Manufacturing Technology}, 2018.

\bibitem{jakeman2019polynomial}
J.~D. Jakeman, F.~Franzelin, A.~Narayan, M.~Eldred, and D.~Plf{\"u}ger,
  ``Polynomial chaos expansions for dependent random variables,''
  \emph{Computer Methods in Applied Mechanics and Engineering}, 2019.

\bibitem{xli2010}
X.~Li, ``Finding deterministic solution from underdetermined equation:
  large-scale performance modeling of analog/{RF} circuits,'' \emph{IEEE Trans.
  CAD}, vol.~29, no.~11, pp. 1661--1668, Nov 2011.

\bibitem{hampton2015compressive}
J.~Hampton and A.~Doostan, ``Compressive sampling of polynomial chaos
  expansions: Convergence analysis and sampling strategies,'' \emph{Journal of
  Computational Physics}, vol. 280, pp. 363--386, 2015.

\bibitem{ma2010adaptive}
X.~Ma and N.~Zabaras, ``An adaptive high-dimensional stochastic model
  representation technique for the solution of stochastic partial differential
  equations,'' \emph{J. Comp. Physics}, vol. 229, no.~10, pp. 3884--3915, 2010.

\bibitem{el2010variation}
T.~El-Moselhy and L.~Daniel, ``Variation-aware interconnect extraction using
  statistical moment preserving model order reduction,'' in \emph{Proc. DATE},
  2010, pp. 453--458.

\bibitem{zhang2014calculation}
Z.~Zhang, T.~A. El-Moselhy, I.~M. Elfadel, and L.~Daniel, ``Calculation of
  generalized polynomial-chaos basis functions and \textsl{G}auss quadrature
  rules in hierarchical uncertainty quantification,'' \emph{IEEE Trans. CAD of
  Integrated Circuits and Systems}, vol.~33, no.~5, pp. 728--740, 2014.

\bibitem{zhang2015enabling}
Z.~Zhang, X.~Yang, I.~Oseledets, G.~E. Karniadakis, and L.~Daniel, ``Enabling
  high-dimensional hierarchical uncertainty quantification by anova and
  tensor-train decomposition,'' \emph{IEEE Trans. CAD of Integrated Circuits
  and Systems}, vol.~34, no.~1, pp. 63--76, 2015.

\bibitem{zhang2017big}
Z.~Zhang, T.-W. Weng, and L.~Daniel, ``Big-data tensor recovery for
  high-dimensional uncertainty quantification of process variations,''
  \emph{IEEE Trans. Components, Packaging and Manufacturing Tech.}, vol.~7,
  no.~5, pp. 687--697, 2017.

\bibitem{zhang2017tensor}
Z.~Zhang, K.~Batselier, H.~Liu, L.~Daniel, and N.~Wong, ``Tensor computation: A
  new framework for high-dimensional problems in {EDA},'' \emph{IEEE Trans. CAD
  Integr. Circuits Syst.}, vol.~36, no.~4, pp. 521--536, 2017.

\bibitem{Walter:1982}
W.~Gautschi, ``On generating orthogonal polynomials,'' \emph{SIAM J. Sci. Stat.
  Comput.}, vol.~3, no.~3, pp. 289--317, Sept. 1982.

\bibitem{tensor:suvey}
T.~G. Kolda and B.~W. Bader, ``Tensor decompositions and applications,''
  \emph{SIAM Rev.}, vol.~51, no.~3, pp. 455--500, 2009.

\bibitem{oseledets2011tensor}
I.~Oseledets, ``Tensor-train decomposition,'' \emph{SIAM Journal Sci. Comp.},
  vol.~33, no.~5, pp. 2295--2317, 2011.

\bibitem{ismail2017review}
M.~E. Ismail and R.~Zhang, ``A review of multivariate orthogonal polynomials,''
  \emph{Journal of the Egyptian Mathematical Society}, vol.~25, no.~2, pp.
  91--110, 2017.

\bibitem{xu1993multivariate}
Y.~Xu, ``On multivariate orthogonal polynomials,'' \emph{SIAM Journal Math.
  Analysis}, vol.~24, no.~3, pp. 783--794, 1993.

\bibitem{barrio2010three}
R.~Barrio, J.~M. Pena, and T.~Sauer, ``Three term recurrence for the evaluation
  of multivariate orthogonal polynomials,'' \emph{Journal of Approximation
  Theory}, vol. 162, no.~2, pp. 407--420, 2010.

\bibitem{Golub:1969}
G.~H. Golub and J.~H. Welsch, ``Calculation of {Gauss} quadrature rules,''
  \emph{Math. Comp.}, vol.~23, pp. 221--230, 1969.

\bibitem{horn2012matrix}
R.~A. Horn, R.~A. Horn, and C.~R. Johnson, \emph{Matrix analysis (Second
  Edition), 7.2 Characterizations and Properties}.\hskip 1em plus 0.5em minus
  0.4em\relax Cambridge university press, 2012.

\bibitem{tracy1993higher}
D.~S. Tracy and S.~A. Sultan, ``Higher order moments of multivariate normal
  distribution using matrix derivatives,'' \emph{Stochastic Analysis and
  Applications}, vol.~11, no.~3, pp. 337--348, 1993.

\bibitem{phillips2010r}
K.~Phillips \emph{et~al.}, ``R functions to symbolically compute the central
  moments of the multivariate normal distribution,'' \emph{Journal of
  Statistical Software, Code Snippets}, vol.~33, no.~1, pp. 1--14, 2010.

\bibitem{oseledets2013constructive}
I.~Oseledets, ``Constructive representation of functions in low-rank tensor
  formats,'' \emph{Constructive Approximation}, vol.~37, no.~1, pp. 1--18,
  2013.

\bibitem{candes2006stable}
E.~J. Cand{\`e}s, J.~K. Romberg, and T.~Tao, ``Stable signal recovery from
  incomplete and inaccurate measurements,'' \emph{Comm. Pure Appl. Math.},
  vol.~59, no.~8, pp. 1207--1223, 2006.

\bibitem{bandeira2013certifying}
A.~S. Bandeira, E.~Dobriban, D.~G. Mixon, and W.~F. Sawin, ``Certifying the
  restricted isometry property is hard,'' \emph{IEEE transactions on
  information theory}, vol.~59, no.~6, pp. 3448--3450, 2013.

\bibitem{buldygin1980sub}
V.~V. Buldygin and Y.~V. Kozachenko, ``Sub-{Gaussian} random variables,''
  \emph{Ukrainian Mathematical Journal}, vol.~32, no.~6, pp. 483--489, 1980.

\bibitem{needell2009cosamp}
D.~Needell and J.~A. Tropp, ``{CoSaMP}: Iterative signal recovery from
  incomplete and inaccurate samples,'' \emph{Appl. Comp. Harm. Analysis},
  vol.~26, no.~3, pp. 301--321, 2009.

\bibitem{zheng2014successive}
X.~Zheng, X.~Sun, D.~Li, and J.~Sun, ``Successive convex approximations to
  cardinality-constrained convex programs: a piecewise-linear dc approach,''
  \emph{Computational Optimization and Applications}, vol.~59, no. 1-2, pp.
  379--397, 2014.

\bibitem{lu2015penalty}
Z.~Lu, Y.~Zhang, and X.~Li, ``Penalty decomposition methods for rank
  minimization,'' \emph{Optimization Methods and Software}, vol.~30, no.~3, pp.
  531--558, 2015.

\bibitem{gu1996efficient}
M.~Gu and S.~C. Eisenstat, ``Efficient algorithms for computing a strong
  rank-revealing {QR} factorization,'' \emph{SIAM J. Sci. Comp.}, vol.~17,
  no.~4, pp. 848--869, 1996.

\bibitem{prasad2016multidimensional}
A.~K. Prasad, M.~Ahadi, and S.~Roy, ``Multidimensional uncertainty
  quantification of microwave/{RF} networks using linear regression and optimal
  design of experiments,'' \emph{IEEE Transactions on Microwave Theory and
  Techniques}, vol.~64, no.~8, pp. 2433--2446, 2016.

\bibitem{diaz2017sparse}
P.~Diaz, A.~Doostan, and J.~Hampton, ``Sparse polynomial chaos expansions via
  compressed sensing and \textsl{D}-optimal design,'' \emph{arXiv preprint
  arXiv:1712.10131}, 2017.

\bibitem{harville1997matrix}
D.~A. Harville, \emph{Matrix algebra from a statistician's perspective}.\hskip
  1em plus 0.5em minus 0.4em\relax Springer, 1997, vol.~1.

\bibitem{sherman1950adjustment}
J.~Sherman and W.~J. Morrison, ``Adjustment of an inverse matrix corresponding
  to a change in one element of a given matrix,'' \emph{The Annals of
  Mathematical Statistics}, vol.~21, no.~1, pp. 124--127, 1950.

\bibitem{malioutov2010sequential}
D.~M. Malioutov, S.~R. Sanghavi, and A.~S. Willsky, ``Sequential compressed
  sensing,'' \emph{IEEE J. Selected Topics in Signal Processing}, vol.~4,
  no.~2, pp. 435--444, 2010.

\bibitem{zhang2018verilog}
K.~Zhang, X.~Xiao, Y.~Zhang, and S.~B. Yoo, ``{Verilog-A} compact modeling and
  simulation of {AWGR} based all-to-all optical interconnects,'' in \emph{CLEO:
  QELS\_Fundamental Science}.\hskip 1em plus 0.5em minus 0.4em\relax OSA, 2018,
  pp. JW2A--49.

\bibitem{petersen1983relation}
L.~Petersen, ``On the relation between the multidimensional moment problem and
  the one-dimensional moment problem,'' \emph{Mathematica Scandinavica}, pp.
  361--366, 1983.

\bibitem{perez2008survey}
D.~P{\'e}rez and Y.~Quintana, ``A survey on the {Weierstrass} approximation
  theorem,'' \emph{Divulgaciones Matem{\'a}ticas}, vol.~16, no.~1, pp.
  231--247, 2008.

\end{thebibliography}

 \begin{IEEEbiography}
  [{\includegraphics[width=1in,height=1.25in,clip,keepaspectratio]{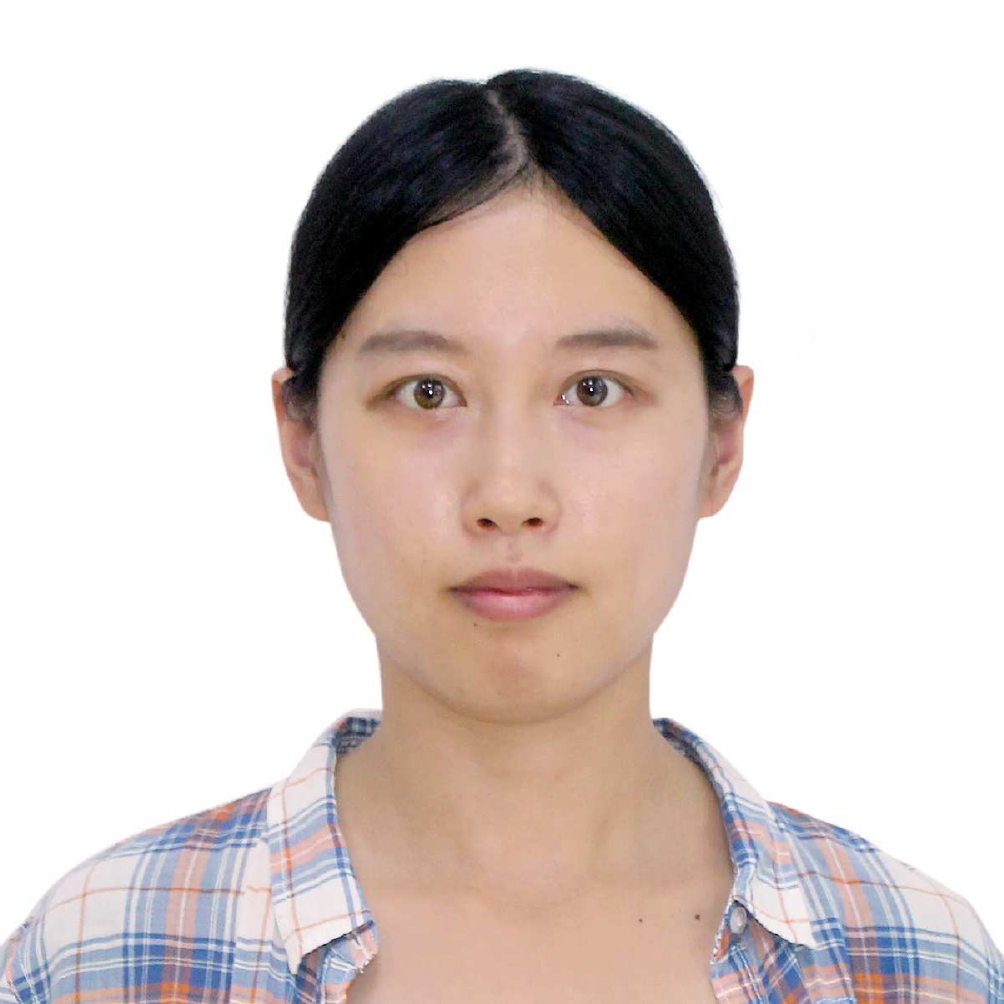}}]{Chunfeng Cui} received the Ph.D. degree in computational mathematics from Chinese Academy of Sciences, Beijing, China, in 2016 with a specialization in numerical optimization. From 2016 to 2017, she was a Postdoctoral Fellow at City University of Hong Kong, Hong Kong. In 2017, she joined the Department of Electrical and Computer Engineering at University of California Santa Barbara as a Postdoctoral Scholar.
From 2011 her research activity is mainly focused in the areas of tensor computations, uncertainty quantification, and machine learning.

She received the Best Paper Award of the IEEE EPEPS 2018.
\end{IEEEbiography}

   \begin{IEEEbiography}[{\includegraphics[width=1in,height=1.25in,clip,keepaspectratio]{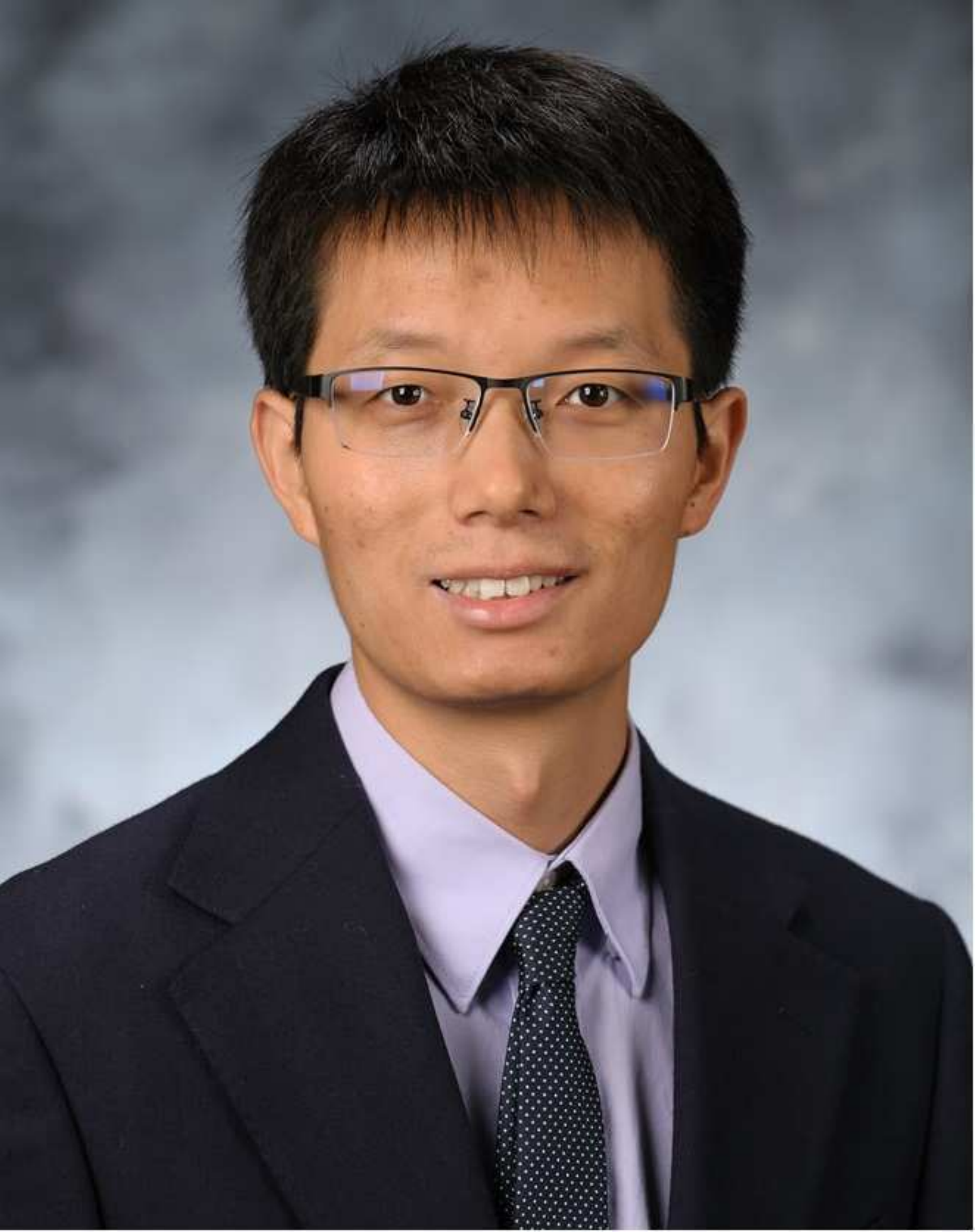}}]{Zheng Zhang} (M'15) received his Ph.D degree in Electrical Engineering and Computer Science from the Massachusetts Institute of Technology (MIT), Cambridge, MA, in 2015. He is an Assistant Professor of Electrical and Computer Engineering with the University of California at Santa Barbara, CA. His research interests include uncertainty quantification and tensor computation, with applications to multi-domain design automation, data analysis and algorithm/hardware co-design of machine learning. 

Dr. Zhang received the Best Paper Award of IEEE Transactions on Computer-Aided Design of Integrated Circuits and Systems in 2014, the Best Paper Award of IEEE Transactions on Components, Packaging and Manufacturing Technology in 2018, two Best Paper Awards (IEEE EPEPS 2018 and IEEE SPI 2016) and three additional Best Paper Nominations (CICC 2014, ICCAD 2011 and ASP-DAC 2011) at international conferences. His Ph.D. dissertation was recognized by the ACM SIGDA Outstanding Ph.D. Dissertation Award in Electronic Design Automation in 2016, and by the Doctoral Dissertation Seminar Award (i.e., Best Thesis Award) from the Microsystems Technology Laboratory of MIT in 2015. He was a recipient of the Li Ka-Shing Prize from the University of Hong Kong in 2011. He received the NSF CAREER Award in 2019.
\end{IEEEbiography}

\end{document}